\documentclass[smallextended]{svjour3}
\usepackage[utf8]{inputenc}
\usepackage{cite}
\usepackage{amsmath}
\usepackage{amsfonts}
\usepackage{amssymb}
\usepackage{subfigure}
\usepackage{enumerate}
\usepackage{hyperref}
\usepackage[colorinlistoftodos]{todonotes}
\usepackage{fullpage}
\newcommand {\R} {\mathbb R}
\newcommand {\B} {\mathbb B}

\newcommand{\al}{\alpha}
\newcommand{\be}{\beta}

\newcommand{\de}{\delta}
\newcommand{\eps}{\varepsilon}
\newcommand{\bx}{\bar x}
\newcommand{\by}{\bar y}

\newcommand {\bd} {{\rm bd}\,}
\newcommand{\iv}{^{-1} }
\newcommand{\tr}{{\rm tr}[A,B](\bx)}
\newcommand{\str}{{\rm str}[A,B](\bx)}
\newcommand{\itr}{{\rm itr}[A,B](\bx)}

\newcommand{\itrd}[1]{{\rm itr}_{#1}[A,B](\bx)}

\newcommand{\strr}{{\rm str}'[A,B](\bx)}
\newcommand{\rg}{{\rm rg}[F](\bx,\by)}
\newcommand{\srg}{{\rm srg}[F](\bx,\by)}
\newcommand {\Int} {{\rm int}\,}
\newcommand {\ri} {{\rm ri}\,}
\newcommand{\ang}[1]{\left\langle #1 \right\rangle}
\newcommand{\qdtx}[1]{\quad\mbox{#1}\quad}
\newcommand\xqed{%
  \leavevmode\unskip\penalty9999 \hbox{}\nobreak\hfill
  \quad\hbox{$\triangle$}}
\def\lsc{lower semicontinuous}

\def\LHS{left-hand side}
\def\RHS{right-hand side}
\def\SVM{set-valued mapping}

\smartqed

\newcommand{\set}[2]{\left\{#1\,\left|\,#2\right.\right\}}

\newcommand{\norm}[1]{\left\|#1\right\|}

\newcommand {\Limsup} {\mathop{{\rm Lim\,sup}\,}}

\DeclareMathOperator{\cone}{{cone}}
\DeclareMathOperator{\dom}{{dom\,}}
\DeclareMathOperator{\gph}{gph}

\newcommand{\sd}{\partial}

\title{About subtransversality of collections of sets
\thanks{AYK was supported by Australian Research Council, project DP160100854. DRL was supported in part by German Israeli Foundation Grant G-1253-304.6 and Deutsche Forschungsgemeinschaft Research Training Grant 2088 TP-B5.
NHT was supported by German Israeli Foundation
Grant G-1253-304.6
}}
\author{
Alexander Y. Kruger
\and
D. Russell Luke
\and
Nguyen~H.~Thao
}
\institute{Alexander Y. Kruger \at
Centre for Informatics and Applied Optimization, Federation University Australia, POB 663, Ballarat, VIC 3350, Australia \\
\email{a.kruger@federation.edu.au}
\and
D. Russell Luke \at
Institut f\"ur Numerische und Angewandte Mathematik, Universit\"at G\"ottingen,
37083 G\"ottingen, Germany\\
\email{r.luke@math.uni-goettingen.de}
\and
Nguyen H. Thao \at
Institut f\"ur Numerische und Angewandte Mathematik, Universit\"at G\"ottingen,
37083 G\"ottingen, Germany.
Department of Mathematics, Teacher College, Can Tho University, Can Tho City, Vietnam\\
\email{h.nguyen@math.uni-goettingen.de,\ nhthao@ctu.edu.vn}
}
\dedication{Dedicated to Professor Michel Th\'era on the occasion of his 70$^{\,th}$ birthday}
\date{Received: date / Accepted: date}
\begin{document}

\maketitle

\begin{abstract}
We provide dual sufficient conditions for subtransversality of collections of sets in an Asplund space setting.
For the convex case, we formulate a necessary and sufficient dual criterion of subtransversality in
general Banach spaces.  Our more general results suggest an intermediate notion of subtransversality,
what we call {\em weak intrinsic subtransversality}, which lies between intrinsic transversality
and subtransversality in Asplund spaces.


\end{abstract}
\section{Introduction}

We study ways several sets in a normed linear space can be arranged in a `regular' way near a point in their intersection.
Such \emph{regular intersection} or, in other words, \emph{transversality} properties are crucial for the validity of \emph{qualification conditions} in optimization as well as subdifferential, normal cone and coderivative calculus, and convergence analysis of computational algorithms.

For brevity, in this article we consider the case of two nonempty sets $A$ and $B$.
The extension of the definitions and characterizations of transversality properties to the case of any finite collection of $n$ sets ($n>1$) does not require much effort (cf. \cite{Kru05.1,Kru06.1,Kru09.1,KruTha13.1,KruTha15}).
The sets are assumed to have a common point $\bar{x}\in A\cap B$.
We shall use the notation $\{A,B\}$ when referring to the pair of two sets $A$ and $B$ as a single object.

The origins of the concept of regular arrangement of sets in space can be traced back to that of \emph{transversality} in differential geometry which deals of course with smooth manifolds (see, for instance, \cite{GuiPol74,Hir76}).
It is motivated by the problem of determining when the intersection of two smooth manifolds is also a smooth manifold near some point in the intersection.
This is true when the collection $\{A,B\}$ of smooth manifolds is
\emph{transversal} at $\bx\in A\cap B $, that is, the sum of the \emph{tangent spaces} to $A$ and $B$ at $\bx$ generates the whole space.
In finite dimensions,
this property can be equivalently characterized in dual terms:
\begin{align}\label{1}
N_{A}(\bar{x}) &\cap N_{B}(\bar{x})= \{0\},
\end{align}
where $N_{A}(\bar{x})$ and $N_{B}(\bar{x})$ are the \emph{normal spaces} (i.e., orthogonal complements to the tangent spaces) to $A$ and $B$, respectively, at the point $\bx$.

In the current article we study arbitrary (not necessarily smooth or convex) sets in a normed linear space and focus on a particular transversality concept, called \emph{subtransversality} which has emerged as a key -- by some estimates \emph{the key} -- notion in the analysis of convergence of iterative methods for solving feasibility problems.
Two equivalent primal space definitions and some qualitative and quantitative characterizations of this property are given in Section~\ref{s:subtrans}, where we also compare subtransversality with a more robust property called simply \emph{transversality} being a generalization of the discussed above corresponding property from differential geometry.

The properties of transversality and subtransversality (also known under many other names) of pairs of sets
correspond directly to \emph{metric regularity} and \emph{metric subregularity} of \SVM s, respectively; see Propositions~\ref{P2}, \ref{P3} and \ref{P2+} below.
This means, in particular, that characterizations of regularity properties of \SVM s can be translated into characterizations of the corresponding transversality properties of pairs of sets and vice versa.
In the current article, when proving characterizations of the subtransversality property of pairs of sets, we follow the sequence proposed in \cite{Kru15} when deducing metric subregularity characterizations for \SVM s.
Characterizations of subtransversality can also be obtained by direct translation of the corresponding statements from \cite{Kru15} using Propositions~\ref{P2} or \ref{P2+}.
We avoid doing this here, first, to keep a self-contained mostly geometrical presentation, and second, because the developments in the current article show that some statements in \cite{Kru15} are formulated not in the strongest form and can be improved.
In fact, the characterizations of subtransversality derived in the current article can be used to improve the corresponding statements in \cite{Kru15}.

In Section \ref{s:dual} we present dual sufficient conditions for subtransversality in Asplund spaces (Theorem~\ref{T1}) as well as a necessary and sufficient criterion for subtransversality of a pair of convex sets in a general Banach space (Theorem~\ref{T2}), and compare them with the corresponding criterion for transversality (Theorem~\ref{T0}).
All three assertions are in a sense analogues (Theorem~\ref{T0} being a direct extension) of the classical criterion \eqref{1}.
Theorem~\ref{T1} extends and strengthens the corresponding assertion announced in the recent paper \cite{KruLukTha}.
Along the way we successively establish several sufficient (and some also necessary) primal and dual conditions of subtransversality and also uncover a new notion of transversality, which we call \emph{weak intrinsic transversality}, that lies somewhere between transversality and subtransversality.
This property as well as a finer property of \emph{intrinsic transversality} (the name is borrowed from \cite{DruIofLew15.2}) are briefly discussed in Section~\ref{s:intrinsic trans}.
A more detailed study of intrinsic transversality and weak intrinsic transversality and their comparison with the corresponding finite dimensional property introduced in \cite{DruIofLew15.2} are going to appear in the forthcoming paper
\cite{Kru}.

\subsection{Notation and preliminaries}

Given a normed linear space $X$,
its topological dual is denoted $X^*$, while $\langle\cdot,\cdot\rangle$ denotes the bilinear form defining the pairing between the spaces.
$\B$ and $\B^*$ stand for the closed unit balls in $X$ and $X^*$, respectively, while $\B_\de(x)$ denotes the open ball with centre at $x\in X$ and radius $\de>0$.
Given a set $A$ in a normed linear space, its interior and boundary are denoted $\Int A$ and $\bd A$, respectively, while $\cone A$ denotes the cone generated by $A$: $\cone A:=\{ta\mid a\in A, t\ge0\}$.
$d_A(x)$ stands for the distance from a point $x$ to a set $A$.
Given an $\al\in\R_\infty:=\R\cup\{+\infty\}$, $\al_+$ denotes its positive part: $\al_+:=\max\{\al,0\}$.
$\mathbb{N}$ is a set of positive integers.

Dual characterizations of transversality and subtransversality properties involve dual space objects -- \emph{normal cones}.
Given a subset $A$ of a normed linear space $X$ and a point $\bx\in A$, the \emph{Fr\'echet normal cone} to $A$ at $\bx$ is defined as follows:
\begin{gather}\label{NC1}
N_{A}(\bx):= \left\{x^*\in X^*\mid \limsup_{a\to\bx,\,a\in A\setminus\{\bx\}} \frac {\langle x^*,a-\bx \rangle}{\|a-\bx\|} \leq 0 \right\}.
\end{gather}
It is a nonempty weak$^*$ closed convex cone, often trivial
($N_{A}(\bx)=\{0\}$).
Similarly,
given a function $f:X\to\R_\infty:=\R\cup\{+\infty\}$ and a point $\bx\in\dom f$, the \emph{Fr\'echet subdifferential} of $f$ at $\bx$ is defined as
\begin{gather}\label{sd}
\sd f(\bx):= \left\{x^*\in X^*\mid \liminf_{x\to\bx,\,x\ne \bx} \frac {f(x)-f(\bx)-\langle x^*,x-\bx \rangle}{\|x-\bx\|} \geq 0 \right\}.
\end{gather}
It is a weak$^*$ closed convex set, often empty.
Using Fr\'echet normal cones, one can define more robust (and in general nonconvex) \emph{limiting normal cones}.
If $\dim X<\infty$, the definition of the \emph{limiting normal cone} to $A$ at $\bx$
takes the following form:
\begin{gather}\label{NC3}
\overline{N}_{A}(\bar x):= \Limsup_{a\to\bx,\,a\in A}N_{A}(a):=\left\{x^*=\lim_{k\to\infty}x^*_k\mid x^*_k\in N_{A}(a_k),\;a_k\in A,\;a_k\to\bx\right\}.
\end{gather}
If $X$ is a Euclidian space and $A$ is closed, the Fr\'echet normal cones in definition \eqref{NC3} can be replaced by the \emph{proximal} ones:
\begin{gather}\label{NC2}
N_{A}^p(\bx):=\cone\left(P_A^{-1}(\bx)-\bx\right).
\end{gather}
Here $P_A$ is the \emph{projection} mapping:
\begin{gather*}
P_{A}(x):=\set{a\in A}{\|x-a\|=d_A(x)},\quad x\in X.
\end{gather*}
It is easy to verify that
$N_{A}^p(\bx)\subset N_{A}(\bx)$,
and $\overline{N}_{A}(\bar x)\ne\{0\}$ if and only if $\bx\in\bd A$.
Unlike \eqref{NC1} and \eqref{NC2}, the cone \eqref{NC3} can be nonconvex.

If $A$ is a convex set, then all three cones \eqref{NC1}, \eqref{NC3} and \eqref{NC2} coincide and reduce to the normal cone in the sense of convex analysis:
\begin{gather*}\label{CNC}
N_{A}(\bx):= \left\{x^*\in X^*\mid \langle x^*,a-\bx \rangle \leq 0 \qdtx{for all} a\in A\right\}.
\end{gather*}

The proofs of the main results rely on two fundamental results of variational analysis: the \emph{Ekeland variational principle} (Ekeland \cite{Eke74}; cf., e.g., \cite[Theorem~2.1]{Kru03.1}, \cite[Theorem 2.26]{Mor06.1}, \cite[Theorem~4B.5]{DonRoc14}) and several kinds of \emph{subdifferential sum rules}.
Below we provide these results for completeness.

\begin{lemma}[Ekeland variational principle] \label{Ek}
Suppose $X$ is a complete metric space, $f:X\to\R_\infty$ is lower semicontinuous and bounded from below, $\varepsilon>0, \lambda>0$. If
$$
f(\bx)<\inf_X f + \varepsilon,
$$
then there exists an $\hat x\in X$ such that

{\rm (a)} $d(\hat x,\bx)<\lambda $,

{\rm (b)} $f(\hat x)\le f(\bx)$,

{\rm (c)} $f(x)+(\varepsilon/\lambda)d(x,\hat x)\ge f(\hat x)$ for all $x\in X$.
\end{lemma}

\begin{lemma}[Subdifferential sum rules] \label{SR}
Suppose $X$ is a normed linear space, $f_1,f_2:X\to\R_\infty$, and $\bx\in\dom f_1\cap\dom f_2$.

{\rm (i) \bf Fuzzy sum rule}. Suppose $X$ is Asplund,
$f_1$ is Lipschitz continuous and
$f_2$
is lower semicontinuous in a neighbourhood of $\bar x$.
Then, for any $\varepsilon>0$, there exist $x_1,x_2\in X$ with $\|x_i-\bar x\|<\varepsilon$, $|f_i(x_i)-f_i(\bar x)|<\varepsilon$ $(i=1,2)$, such that
$$
\partial (f_1+f_2) (\bar x) \subset \partial f_1(x_1) +\partial f_2(x_2) + \varepsilon\B^\ast.
$$


{\rm (ii) \bf Convex sum rule}. Suppose
$f_1$ and $f_2$ are convex and $f_1$ is continuous at a point in $\dom f_2$.
Then
$$
\partial (f_1+f_2) (\bar x) = \sd f_1(\bx) +\partial f_2(\bx).
$$
\end{lemma}

The first sum rule in the lemma above is known as the \emph{fuzzy} or \emph{approximate} sum rule (Fabian \cite{Fab89}; cf., e.g., \cite[Rule~2.2]{Kru03.1}, \cite[Theorem~2.33]{Mor06.1}) for Fr\'echet subdifferentials in Asplund spaces.
The other one is an example of an \emph{exact} sum rule.
It is valid in arbitrary normed (or even locally convex) spaces.
For rule (ii) we refer the readers to \cite[Theorem~0.3.3]{IofTik79} and
\cite[Theorem~2.8.7]{Zal02}.

Recall that a Banach space is \emph{Asplund} if every continuous convex function on an open convex set is Fr\'echet differentiable on some its dense subset \cite{Phe89}, or equivalently, if the dual of each its separable subspace is separable.
We refer the reader to \cite{Phe89,Mor06.1,BorZhu05} for discussions about and characterizations of Asplund spaces.
All reflexive, in particular, all finite dimensional Banach spaces are Asplund.

\section{Transversality and subtransversality}\label{s:subtrans}

In this introductory section we briefly discuss two standard regularity properties of a pair of sets in a normed linear space, namely \emph{transversality} and \emph{subtransversality} (also known under other names) with the emphasis on the second one.

\begin{definition}\label{D1}
Suppose $X$ is a normed linear space, $A,B\subset X$, and $\bx\in A\cap B$.
$\{A,B\}$ is \emph{subtransversal} at $\bar x$ if one of the following two equivalent conditions is satisfied:
\begin{enumerate}
\item
there exist numbers $\alpha\in]0,1[$ and $\delta>0$ such that
\begin{equation}\label{D1-1}
\bigl(A+(\alpha\rho)\B\bigr)\cap \bigl(B+(\alpha\rho)\B\bigr)\cap\B_{\delta}(\bar x)\subset
\left(A\cap B\right)+\rho\B
\;\;\mbox{for all}\;\; \rho\in ]0,\delta[;
\end{equation}
\item
there exist numbers $\alpha\in]0,1[$ and $\delta>0$ such that
\begin{gather}\label{D1-2}
\alpha d\left(x,A\cap B\right)\le \max\left\{d(x,A),d(x,B)\right\}
\quad\mbox{for all}\quad x\in \B_{\delta}(\bar{x}).
\end{gather}
\end{enumerate}
The exact upper bound of all $\alpha\in]0,1[$ such that condition \eqref{D1-1} or condition \eqref{D1-2} is satisfied for some $\de>0$ is denoted $\str$ with the convention that the supremum of the empty subset of $\R_+$ equals~0.
\end{definition}

The requirement that $\al<1$ in both parts of Definition~\ref{D1} imposes no restrictions on the property.
It is only needed in the case $\bx\in\Int(A\cap B)$ (when conditions \eqref{D1-1} and \eqref{D1-2} are satisfied for some $\de>0$ with any $\al>0$) to ensure that $\str$ is always less than or equal to 1 and simplify the subsequent quantitative estimates.
It is easy to check that when $\bx\in\bd(A\cap B)$, each of the conditions \eqref{D1-1} and \eqref{D1-2} implies $\al\le1$.
We are going to use similar requirements in other definitions throughout the article.

The property in part (i) of Definition~\ref{D1} was introduced recently in \cite{KruTha15} (under the name \emph{subregularity}).
It can be viewed as a local analogue of the global \emph{uniform normal property} introduced in the convex setting in \cite[Definition~3.1(4)]{BakDeuLi05} as a generalization of the \emph{property (N)} of convex cones by Jameson \cite{Jam72}.
A particular case of the Jameson property (N) for convex cones $A$ and $B$ such that $B=-A$ and $A\cap(-A)=\{0\}$ was studied by M. Krein in the 1940s.
Subtransversality constant $\str$ is, in a sense, a local analogue of the \emph{normality constant} in \cite[Definition~4.2]{BakDeuLi05}.

The metric property in part (ii) of Definition~\ref{D1} is a very well known regularity property that has been around for more than 30 years under various names ((local) \emph{linear regularity}, \emph{metric regularity}, \emph{linear coherence}, \emph{metric inequality}, and \emph{subtransversality}); cf. \cite{BakDeuLi05,
BauBor93,Dol82,
RocWet98,BauBor96, Iof89,Iof00_,Iof16,KlaLi99,HesLuk13,
LiNgPon07,NgaiThe01,Pen13,ZheNg08,ZheWeiYao10, DruIofLew15.2}.
It has been used as the key assumption when establishing linear convergence of sequences generated by alternating projection algorithms and a qualification condition for subdifferential and normal cone calculus formulae.
One can also observe that condition \eqref{D1-2} is equivalent to the function $x\mapsto\max\{d(x,A),d(x,B)\}$ having a local \emph{error bound} \cite{Aze03,FabHenKruOut10,Kru15}/\emph{weak sharp minimum} \cite{BurDeng02,BurDeng05,BurFer93} at $\bx$ with constant $\al$.
The equivalence of the two properties in Definition~\ref{D1} and the fact that the exact upper bounds of all $\alpha\in]0,1[$ in conditions \eqref{D1-1} and \eqref{D1-2} coincide were established in \cite[Theorem~3.1]{KruTha15}.

The subtransversality of $\{A,B\}$ is equivalent to the condition $\str>0$, and $\str$ provides a quantitative characterization of this property.

The subtransversality property of pairs of sets in Definition~\ref{D1} is a weaker version of another well known regularity property in the next definition.

\begin{definition}\label{D2+}
Suppose $X$ is a normed linear space, $A,B\subset X$, and $\bx\in A\cap B$.
$\{A,B\}$ is \emph{transversal} at $\bar x$ if one of the following two equivalent conditions is satisfied:
\begin{enumerate}
\item
there exist numbers $\alpha\in]0,1[$ and $\delta>0$ such that
\begin{equation}\label{D2-2}
(A-a-x_1)\cap(B-b-x_2)\cap(\rho\B )\neq \emptyset
\end{equation}
for all $\rho \in ]0,\delta[$, $a\in A\cap\B_{\delta}(\bar{x})$, $b\in B\cap\B_{\delta}(\bar{x})$, and all $x_1,x_2\in X$ with $\max\{\|x_1\|,\|x_2\|\}<\alpha\rho$;
\item
there exist numbers $\alpha\in]0,1[$ and $\delta>0$ such that
\begin{gather}\label{D2-3}
\alpha d\left(x,(A-x_1)\cap (B-x_2)\right)\le \max\left\{d(x,A-x_1),d(x,B-x_2)\right\}
\;\;\mbox{for all}\;\; x\in \B_{\delta}(\bar{x}),\;x_1,x_2\in \delta\B.
\hspace{-.2cm}
\end{gather}
\end{enumerate}
The exact upper bound of all $\alpha\in]0,1[$ such that condition \eqref{D2-2} or condition \eqref{D2-3} is satisfied for some $\de>0$ is denoted $\tr$ with the convention that the supremum of the empty subset of $\R_+$ equals 0.
\end{definition}

The property in part (i) of Definition~\ref{D2+} was introduced by the first author in 2005.
Since then the terminology in the papers (co-)authored by him has changed several time causing some confusion, for which he apologizes to the readers.
The next table reflects the evolution of the terminology.
\begin{center}
\begin{tabular}{|c|c|c|c|c|}
  \hline
  2005 \cite{Kru05.1} & 2006 \cite{Kru06.1} & 2009 \cite{Kru09.1} & 2013 \cite{KruTha13.1} & 2017 \cite{KruLukTha} \\
  \hline
  Regularity & Strong regularity & Property (UR)$_S$ & Uniform regularity & Transversality \\
  \hline
\end{tabular}
\end{center}
In \cite{LewLukMal09} the property is called \emph{linearly regular intersection}.
If $A$ and $B$ are closed convex sets and $\Int A\ne\emptyset$, then this property is equivalent to the conventional qualification condition: $\Int A\cap B\ne\emptyset$ (cf. \cite[Proposition~14]{Kru05.1}).

The metric property in part (ii) of Definition~\ref{D2+} was referred to in \cite{Kru05.1,Kru06.1,Kru09.1} as \emph{strong metric inequality}.
The equivalence of the two properties in Definition~\ref{D2+} and the fact that the exact upper bounds of all $\alpha\in]0,1[$ in conditions (i) and (ii) coincide were established in \cite[Theorem~1]{Kru05.1}.

From comparing the second parts in Definitions~\ref{D1} and \ref{D2+}, one can see that the transversality of a pair of sets corresponds to the subtransversality of all their small translations holding uniformly (cf. \cite[p.~1638]{DruIofLew15.2}).
The next inequality is straightforward:
$$\tr\le\str.$$

\begin{example}\label{E3}
If $A=B$, then $d\left(x,A\cap B\right)=d(x,A)=d(x,B)$ for any $x\in X$.
Hence, condition \eqref{D1-2} holds with any  $\alpha\in]0,1[$ and $\delta>0$.
Thus, $\{A,B\}$ is subtransversal at $\bar x$ and $\str=1$.
\xqed\end{example}

Note that, under the conditions of Example~\ref{E3}, $\{A,B\}$ does not have to be transversal at $\bar x$.

\begin{example}\label{E1}
Let $ X=\R^2$, $A=B=\R\times\{0\}$, and $\bx=(0,0)$.
If $x_1=(0,\eps)$ and $x_2=(0,0)$, then condition \eqref{D2-2} does not hold for any $a\in A$, $b\in B$, $\rho>0$, and $\eps>0$.
Thus, $\{A,B\}$ is subtransversal at $\bar x$ thanks to Example~\ref{E3}, but not transversal, and $\tr=0$.
\xqed\end{example}

We refer the reader to \cite{KruTha15} for more examples illustrating the relationship between the properties in Definitions~\ref{D1} and \ref{D2+}.

The next proposition provides a useful metric characterization of the subtransversality property complementing the one in part (ii) of Definition~\ref{D1}.
It was established in \cite[Theorem~1(iii)]{KruLukTha} in the Euclidean space setting, but the proof given there is valid in an arbitrary normed linear space.

\begin{proposition}\label{P1}
Suppose $X$ is a normed linear space, $A,B\subset X$, and $\bx\in A\cap B$.
$\{A,B\}$ is \emph{subtransversal} at $\bar x$  if and only if there exist numbers $\alpha\in]0,1[$ and $\delta>0$ such that
\begin{equation}\label{P3-1}
\alpha d(x,A\cap B)\le d(x,B)\;\;\mbox{for all}\;\; x\in A\cap\B_{\delta}(\bar{x}).
\end{equation}
Moreover,
\begin{equation}\label{P3-2}
\frac{1}{2(\strr)\iv+1}\le\str\le\strr,
\end{equation}
where $\strr$ is the exact upper bound of all numbers $\alpha\in]0,1[$ such that condition (\ref{P3-1}) is satisfied, with the convention that the supremum of the empty subset of $\R_+$ equals 0.
\end{proposition}

Proposition~\ref{P1} can be considered as a nonconvex extension of \cite[Theorem~3.1]{NgYang04}.

\begin{remark}
1. The maximum of the distances in Definitions~\ref{D1} and \ref{D2+} (explicitly present in part (ii) and implicitly also in part (i)) and some other representations in the sequel corresponds to the maximum norm in $\R^2$ employed in all these definitions and assertions.
It can be replaced everywhere by the sum norm (pretty common in this type of definitions in the literature) or any other equivalent norm.
All the assertions above including the quantitative characterizations will remain valid (as long as the same norm is used everywhere), although the exact values of $\str$ and $\tr$ do depend on the chosen norm and some estimates (e.g. in Propositions~\ref{P1}) can change.

2. In some situations it can be convenient to use the reciprocal $(\str)\iv$ instead of $\str$ for characterizing the subtransversality property.
The property is obviously equivalent to $(\str)\iv<\infty$.
For instance, using the reciprocals, the quantitative estimates (\ref{P3-2}) in Propositions~\ref{P1} can be rewritten in a simpler form as
\sloppy
\begin{equation*}
(\strr)\iv\le(\str)\iv\le2(\strr)\iv+1.
\end{equation*}

3. Thanks to Propositions~\ref{P1}, one can use $\strr$ instead of $\str$ for quantitative characterization of the subtransversality property.
Note that $\strr$ is not symmetric: ${\rm str}'[B,A](\bx)\ne\strr$.
One can strengthen the conclusion of Propositions~\ref{P1} by replacing $\strr$ in the \RHS\ of (\ref{P3-2}) by $\min\{\strr,{\rm str}'[B,A](\bx)\}$ and by $\max\{\strr,{\rm str}'[B,A](\bx)\}$ in its \LHS.
\xqed\end{remark}

Not surprisingly, transversality properties of pairs of sets 
are strongly connected with the corresponding regularity properties of \SVM s.
The properties in Definitions~\ref{D1} and \ref{D2+} correspond, respectively, to \emph{metric subregularity} and \emph{metric regularity} of \SVM s (cf., e.g., \cite{DonRoc14}), which partially explains the terminology adopted in the current article.

\begin{definition}\label{MR}
Suppose $X$ and $Y$ are metric spaces, $F:X \rightrightarrows Y$, and $(\bx,\by)\in\gph F:=\{(x,y)\in X\times Y\mid y\in F(x)\}$.
\begin{enumerate}
\item
$F$ is \emph{metrically regular} at $(\bx,\by)\in\gph F$ if there exist numbers $\alpha>0$ and $\delta>0$ such that
\begin{equation*}
\alpha d\left(x,F^{-1}(y)\right) \le d(y,F(x))\;\;\mbox{for all}\;\; x \in \B_{\delta}(\bar{x}),\;
y \in \B_{\delta}(\bar{y});
\end{equation*}
\item
$F$ is \emph{metrically subregular} at $(\bx,\by)\in\gph F$ if there exist numbers $\alpha>0$ and $\delta>0$ such that
\begin{equation*}
\alpha d\left(x,F^{-1}(\by)\right) \le d(\by,F(x))\;\;\mbox{for all}\;\; x \in \B_{\delta}(\bar{x}).
\end{equation*}
\end{enumerate}
\end{definition}
In a slight violation of the notation adopted in \cite{DonRoc14}, we will use $\rg$ and $\srg$ to denote the exact upper bounds of all $\al$ in parts (i) and (ii) of the above definition, respectively.

The regularity properties in Definition~\ref{MR} lie at the core of the contemporary variational analysis.
They have their roots in classical analysis and are crucial for the study of stability of solutions to (generalized) equations and various aspects of subdifferential calculus and optimization theory.
For the state of the art of the regularity theory of \SVM s and its numerous applications we refer the reader to the book by Dontchev and Rockafellar \cite{DonRoc14} and the comprehensive survey by Ioffe \cite{Iof16,Iof2}.

Given a pair of subsets $A$ and $B$ of a normed linear space $X$, one can define a \SVM\ $F:X \rightrightarrows X^2$ by the equality (cf. \cite{Iof00_,Iof16})
\begin{equation}\label{7}
F(x):=(A-x)\times(B-x), \quad x\in X.
\end{equation}
The next proposition employs the maximum norm on $X^2$ ($\norm{(x_1,x_2)}:=\max\{\norm{x_1},\norm{x_2}\}$, $x_1,x_2\in X$).

\begin{proposition}\label{P2}
Suppose $X$ is a normed linear space, $A,B\subset X$, $\bx\in A\cap B$, and a \SVM\ $F:X \rightrightarrows X^2$ is defined by \eqref{7}.
\begin{enumerate}
\item
$\{A,B\}$ is {transversal} at $\bar x$ if and only if $F$ is {metrically regular} at $(\bar x,0)$;
\item
$\{A,B\}$ is {subtransversal} at $\bar x$ if and only if $F$ is {metrically subregular} at $(\bar x,0)$.
\end{enumerate}
Moreover, $\tr={\rm rg}[F](\bx,0)$ and $\str={\rm srg}[F](\bx,0)$.
\end{proposition}

Conversely, given a \SVM\ $F:X\rightrightarrows Y$ between normed linear spaces and a point $(\bx,\by)\in\gph F$, one can define two sets in $X\times Y$:
\begin{equation}\label{8}
A:=\gph F,\quad B:=X\times \{\bar y\}.
\end{equation}
The next proposition employs the maximum norm on $X\times Y$ ($\norm{(x,y)}:=\max\{\norm{x},\norm{y}\}$, $x\in X$, $y\in Y$).

\begin{proposition}\label{P3}
Suppose $X$ and $Y$ are normed linear spaces, $F:X\rightrightarrows Y$, $(\bx,\by)\in\gph F$, and sets $A$ and $B$ are defined by \eqref{8}.
\begin{enumerate}
\item
$F$ is {metrically regular} at $(\bar x,\by)$ if and only if $\{A,B\}$ is {transversal} at $(\bx,\by)$; \item
$F$ is {metrically subregular} at $(\bar x,\by)$ if and only if $\{A,B\}$ is subtransversal at $(\bx,\by)$.
\end{enumerate}
Moreover,
\begin{gather*}
\frac{1}{2(\rg)\iv+1} \le\tr\le\min\left\{\frac{\rg}{2},1\right\},
\\
\frac{1}{2(\srg)\iv+1} \le\str\le\min\left\{\frac{\srg}{2},1\right\}.
\end{gather*}
\end{proposition}

The equivalences in Propositions~\ref{P2} and \ref{P3} and some quantitative estimates can be found in \cite[Theorems~2 and 3, and Corollaries~2.1 and 3.1]{Kru05.1}; see also \cite[Proposition~3.5]{Iof00_}, \cite[Theorem~6.12]{Iof16}, \cite[Propositions~8 and 9]{Kru06.1}, \cite[Theorems~7 and 8, and Corollary~7.1]{Kru09.1} and \cite[Theorem~3]{KruLukTha}.
The quantitative estimates in Proposition~\ref{P3} are taken from \cite[Theorem~5.1]{KruTha15}.

\begin{remark}
The quantitative estimates in Proposition~\ref{P3} can be improved by choosing an appropriate norm on $X\times Y$.
\end{remark}

In the Euclidian space setting, the following (not more than) single-valued mapping $G:X^2\rightrightarrows X$ can replace \eqref{7} in the equivalences in Proposition~\ref{P2} (cf. \cite{LewMal08}):
\begin{equation}\label{9}
G(x_1,x_2):=
\begin{cases}
\{x_1-x_2\}&\mbox{if } x_1\in A\mbox{ and }x_2\in B,
\\
\emptyset&\mbox{otherwise}.
\end{cases}
\end{equation}

The next proposition employs the Euclidian norm on $X^2$ ($\norm{(x_1,x_2)}:=\sqrt{\norm{x_1}^2+\norm{x_2}^2}$, $x_1,x_2\in X$).

\begin{proposition}\label{P2+}
Suppose $X$ is a Euclidian space, $A,B\subset X$, $\bx\in A\cap B$, and a mapping $G:X^2\rightrightarrows X$ is defined by \eqref{9}.
\begin{enumerate}
\item
$\{A,B\}$ is {transversal} at $\bar x$ if and only if $G$ is {metrically regular} at $((\bar x,\bar x),0)$;
\item
$\{A,B\}$ is {subtransversal} at $\bar x$ if and only if $G$ is {metrically subregular} at $((\bar x,\bar x),0)$.
\end{enumerate}
Moreover,
\begin{gather*}
\frac{1}{2({\rm rg}[G]((\bx,\bx),0))\iv+1}\le\tr
\le\frac{1}{\sqrt{2({\rm rg}[G]((\bx,\bx),0))^{-2}-1}},
\\
\frac{1}{2({\rm srg}[G]((\bx,\bx),0))\iv+1}\le\str
\le\frac{1}{\sqrt{2({\rm srg}[G]((\bx,\bx),0))^{-2}-1}}.
\end{gather*}
\end{proposition}

The above proposition is extracted from \cite[Theorem~3]{KruLukTha}; see also \cite[Corollary~6.13]{Iof16}.

In view of Propositions~\ref{P2}, \ref{P3} and \ref{P2+}, regularity models in terms of \SVM s and pairs of sets are in a sense equivalent.
In the current article we focus on the second model.

One of the typical applications of transversality properties of pairs (or more generally finite collections) of sets is to the convergence analysis of \emph{alternating (or cyclic) projections} for solving \emph{feasibility problems} \cite{Bre65,GubPolRai67,BauBor93,BauBor96, LewMal08,LewLukMal09, AttBolRedSou10,BauLukPhaWan13.1,BauLukPhaWan13.2, HesLuk13,DruIofLew15.2,NolRon16,KruTha16,
Iof2,KruLukTha,LukThaTeb}.

Given two sets $A$ and $B$, the feasibility problem consists in finding a point in their intersection $A\cap B$.
If these are closed sets in finite dimensions, alternating projections are determined by a sequence $(x_k)$ starting with some point $x_0$ and such that
\begin{equation*}
x_{k+1}\in P_AP_B(x_{k})\quad(k=0,1,\ldots).
\end{equation*}
Here $P_A$ and $P_B$ stand for the Euclidean projection operators on the corresponding sets, i.e., e.g., $$P_A(x):=\{a\in A\mid \norm{x-a}=d(x,A)\},$$
where the Euclidean norm and distance are used.
If $A$ is closed and convex, then $P_A$ is a singleton.
In analyzing convergence of the alternating projections $(x_k)$, it is usually helpful to look at the sequence of intermediate points $(b_k)$ with
$b_k\in P_B(x_{k})$ and $x_{k+1}\in P_A(b_k)$ ($k=0,1,\ldots$).
We denote the joining sequence by $(z_k)$, that is
\begin{align}\label{z_k}
z_{2n} = x_n \mbox{ and } z_{2n+1} = b_n,\quad (n=0,1,\ldots).
\end{align}
For simplicity of presentation let us assume throughout the discussion, without loss of generality, that $x_0\in A$.


Bregman \cite{Bre65} and Gubin et al \cite{GubPolRai67} showed that, if $A\cap B\ne\emptyset$ and the sets are closed and convex, the sequence converges to a point in $A\cap B$.
In the case of two subspaces, this fact was established by von Neumann in the mid-1930s; that is why the method of alternating projections is sometimes referred to as \emph{von Neumann’s method}.
It was noted in \cite{NolRon16} that alternating projections can be traced back to the 1869 work by Schwarz.
It was shown in \cite{GubPolRai67} that, if $\ri A\cap\ri B\ne\emptyset$, the convergence is linear, i.e.,
\begin{gather}\label{11}
\norm{x_k-\hat x}\le\al c^k\quad(k=0,1,\ldots),
\end{gather}
where $\hat x\in A\cap B$ is the limit of the sequence, $\al>0$ and $c\in]0,1[$.
If \eqref{11} holds, it is often said that $(x_k)$ converges with \emph{$R$-linear rate} $c$.
A systematic analysis of the convergence of alternating projections in the convex setting was done by Bauschke and Borwein \cite{BauBor93,BauBor96}, who demonstrated that it is the subtransversality property in Definition~\ref{D1} that is needed to ensure linear convergence.
In fact, as the next proposition taken from \cite{LukThaTeb} shows, subtransversality in the convex setting is necessary and sufficient for linear convergence of alternating projections.

\begin{proposition}\label{P5-}
Suppose $X$ is a Euclidean space, $A,B\subset X$ are closed and convex, and $\bx\in A\cap B$.
\begin{enumerate}
\item
If
$\{A, B\}$ is subtransversal at $\bx$, then alternating projections converge linearly with rate at most
$1-\str^2$,
provided that the starting point is sufficiently close to $\bx$.
\item
If alternating projections converge linearly with rate $c\in]0,1[$ for any starting point sufficiently close to $\bx$, then $\{A, B\}$ is subtransversal at $\bx$ and $\str\ge\frac{1-c}{3-c}$.
%
\end{enumerate}
\end{proposition}

The picture becomes much more complicated
if the convexity assumption is dropped.
In view of the following proposition taken from \cite{LukThaTeb}, subtransversality remains a necessary condition for certain types of linear convergence of alternating projections.

\begin{proposition}\label{Nec_of_SubTra}
Suppose $X$ is a Euclidean space, $A,B\subset X$ are closed, and $\bx\in A\cap B$.
If for any starting point $x_0$ sufficiently close to $\bx$, \begin{enumerate}
\item
either every sequence of alternating projections $(x_k)$
is linear monotone with rate $c\in]0,1[$ in the sense 
that
\begin{gather}\label{e:lin mon}
d(x_{k+1},A\cap B)\le c d(x_k,A\cap B) \quad(k=0,1,\ldots),
\end{gather}
\item
or every sequence of joining alternating projections
$(z_k)$ determined by \eqref{z_k} satisfies the following conditions for a constant $c\in ]0,1[$
\begin{subequations}\label{e:lin extend}
\begin{align}\label{LinCon_Seq+'_m}
&\|z_{k+2}-z_{k+1}\|\; \le\; \|z_{k+1}-z_{k}\|, \quad(k=0,1,\ldots),
\\
\label{LinCon_Seq+'_m2}
&\|z_{2k+2}-z_{2k+1}\|\; \le\; c\|z_{2k+1}-z_{2k}\|, \quad(k=0,1,\ldots),
\end{align}
\end{subequations}
\end{enumerate}
then $\{A, B\}$ is subtransversal at $\bx$ and $\str\ge\frac{1-c}{5-c}$.
\end{proposition}

As shown in \cite{LukThaTeb}, properties \eqref{e:lin mon} and
\eqref{e:lin extend} both
imply linear convergence of alternating projections with $R$-linear rate,
and the three properties are equivalent when the sets are convex.
It is conjectured in \cite{LukThaTeb} that subtransversality is necessary for linear convergence of (both convex and nonconvex) alternating projections.
At the same time, simple examples show that subtransversality is not sufficient to guarantee (any) convergence of alternating projections to a solution of the feasibility problem.

A study of the convergence of alternating
projections in the nonconvex setting was initiated recently by Lewis and Malick \cite{LewMal08}, and Lewis et al. \cite{LewLukMal09}, who demonstrated in the Euclidean space setting that a stronger transversality property in Definition~\ref{D2+} is sufficient for the local linear convergence of alternating projections for, respectively, a pair of smooth manifolds or a pair of arbitrary closed sets one of which is \emph{super-regular} at the reference point.
The last property holds, in particular, for convex sets and smooth manifolds.
It was shown later by Drusvyatskiy et al. \cite{DruIofLew15.2} that transversality guarantees local
linear convergence of alternating projections for a pair of closed sets in a Euclidean space without the super-regularity assumption.
The role of the transversality property in the convergence analysis of alternating projections in the nonconvex setting is studied in Drusvyatskiy et al. \cite{DruIofLew15.2}, Kruger and Thao \cite{KruTha16}, Noll and Rondepierre \cite{NolRon16}, and Kruger et al. \cite{KruLukTha}.

In view of Propositions~\ref{P5-} and \ref{Nec_of_SubTra} and the above discussion, subtransversality is close to being necessary for the local linear convergence of alternating projections for a pair of closed sets in a Euclidean space, but is not sufficient unless the sets are convex.
On the other hand, transversality is sufficient, but is far from being necessary even in the convex case.
For example, transversality always fails when the affine span of the union of the sets is not equal to the whole space, while alternating projections can still converge linearly as is the case when the sets are convex with nonempty intersection of their relative interiors.
A quest has started for the weakest regularity property lying between transversality and subtransversality and being sufficient for the local linear convergence of alternating projections.
We mention here the articles by Bauschke et al. \cite{BauLukPhaWan13.1,BauLukPhaWan13.2} utilizing \emph{restricted normal cones}, Drusvyatskiy et al. \cite{DruIofLew15.2} introducing and successfully employing \emph{intrinsic transversality}, Noll and Rondepierre \cite{NolRon16} introducing a concept of \emph{separable intersection}, with 0-separability being a weaker property than intrinsic transversality and still implying the local linear convergence of alternating projections under the additional assumption that one of the sets is 0-\emph{H\"older regular} at the reference point with respect to the other.

\section{Dual characterizations}\label{s:dual}

The dual criterion for the transversality property in Definition~\ref{D2+} in Asplund spaces is well known, see \cite{Kru05.1,Kru06.1,Kru09.1,KruTha13.1,KruTha15}.

\begin{theorem}\label{T0}
Suppose $X$ is Asplund, $A,B\subset X$ are closed, and $\bx\in A\cap B$. Then
$\{A,B\}$ is transversal at $\bar x$ if and only if there exist numbers $\alpha\in]0,1[$ and $\delta>0$ such that
$\|x^*_1+x^*_2\|>\alpha$ for all $a\in A\cap\B_\delta(\bx)$, $b\in B\cap\B_\delta(\bx)$, and all $x_1^*\in N_{A}(a)$ and $x_2^*\in N_{B}(b)$ satisfying $\|x^*_1\|+\|x^*_2\|=1$.
Moreover, the exact upper bound of all such $\al$ equals $\tr$.
\end{theorem}

In finite dimensions, the above criterion admits convenient equivalent reformulations in terms of limiting normals.

\begin{corollary}\label{C00}
Suppose $\dim X<\infty$, $A,B\subset X$ are closed, and $\bx\in A\cap B$.
Then $\{A,B\}$ is transversal at $\bar x$ if and only if one of the following two equivalent conditions is satisfied:
\begin{enumerate}
\item
there exists a number $\alpha\in]0,1[$ such that ${\|x^*_1+x^*_2\|>\alpha}$
for all $x^*_1\in\overline{N}_{A}(\bar x)$ and $x^*_2\in\overline{N}_{B}(\bar x)$
satisfying
$\|x^*_1\|+\|x^*_2\|=1$;
\item
$\overline{N}_{A}(\bar x)\cap\left(-\overline{N}_{B}(\bar x)\right)=\{0\}$.
\end{enumerate}
Moreover, the exact upper bound of all $\alpha$ in {\rm (i)} equals $\tr$.
\end{corollary}

The property in part (ii) of Corollary~\ref{C00} is a well known qualification condition/nonseparabilty property that has been around for about 30 years under various names (\emph{basic qualification condition}, \emph{normal qualification condition}, \emph{transversality}, \emph{transversal intersection}, \emph{regular intersection}, \emph{linearly regular intersection}, and \emph{alliedness property}); cf. \cite{Mor88,Mor06.1,ClaLedSteWol98,Pen13,LewMal08, LewLukMal09,Iof16}.
When $A$ and $B$ are smooth manifolds, it coincides with \eqref{1}.

The next theorem deals with the subtransversality property in Definition~\ref{D1}.
It provides a dual sufficient condition for this property in an Asplund space.

\begin{theorem}\label{T1}
Suppose $X$ is Asplund, $A,B\subset X$ are closed, and $\bx\in A\cap B$. Then
$\{A,B\}$ is subtransversal at $\bar x$ if there exist numbers $\alpha\in]0,1[$ and $\delta>0$ such that,
for all $a\in(A\setminus B)\cap\B_\de(\bx)$, $b\in(B\setminus A)\cap\B_\de(\bx)$ and $x\in\B_\de(\bx)$ with $\norm{x-a}=\norm{x-b}$,
there exists an $\eps>0$ such that
$\|x^*_1+x^*_2\|>\alpha$ for all $a'\in A\cap\B_\eps(a)$, $b'\in B\cap\B_\eps(b)$, $x_1'\in\B_\eps(a)$, $x_2'\in\B_\eps(b)$, $x'\in\B_\eps(x)$, and $x_1^*,x_2^*\in X^*$
satisfying
\begin{gather}\label{T1-1}
\norm{x'-x_1'}=\norm{x'-x_2'},
\\\label{T1-2}
\|x^*_1\|+\|x^*_2\|=1,\quad
\ang{x^*_1,x'-x_1'}=\|x^*_1\|\|x'-x_1'\|,\quad
\ang{x^*_2,x'-x_2'}=\|x^*_2\|\|x'-x_2'\|,
\\\label{T1-3}
d(x_1^*,N_{A}(a'))<\delta,\quad d(x_2^*,N_{B}(b'))<\delta.
\end{gather}
Moreover, $\str\ge\al$.
\end{theorem}

In the convex case, one can formulate a necessary and sufficient dual criterion of subtransversality in general Banach spaces which takes a simpler form.

\begin{theorem}\label{T2}
Suppose $X$ is a Banach space, $A,B\subset X$ are closed and convex, and $\bx\in A\cap B$.
Then $\{A,B\}$ is subtransversal at $\bar x$ if and only if there exist numbers $\alpha\in]0,1[$ and $\delta>0$ such that $\|x^*_1+x^*_2\|>\alpha$
for all $a\in(A\setminus B)\cap\B_\de(\bx)$, $b\in(B\setminus A)\cap\B_\de(\bx)$, $x\in\B_\de(\bx)$ with $\norm{x-a}=\norm{x-b}$, and $x_1^*,x_2^*\in X^*$ satisfying
\begin{gather}\label{T2-1}
\|x^*_1\|+\|x^*_2\|=1,\quad
\ang{x^*_1,x-a}=\|x^*_1\|\|x-a\|,\quad
\ang{x^*_2,x-b}=\|x^*_2\|\|x-b\|,
\\\notag
d(x_1^*,N_{A}(a))<\delta,\quad d(x_2^*,N_{B}(b))<\delta.
\end{gather}
Moreover, the exact upper bound of all such $\al$ equals $\str$.
\end{theorem}

\begin{remark}
1. It is sufficient to check the conditions of  Theorems~\ref{T0}, \ref{T1} and \ref{T2} only for $x^*_1\ne0$ and $x^*_2\ne0$.
Indeed, if one of the vectors $x^*_1$ and $x^*_2$ equals 0, then by the normalization condition $\|x^*_1\|+\|x^*_2\|=1$, the norm of the other one equals 1, and consequently $\|x^*_1+x^*_2\|=1$, i.e., such pairs $x^*_1,x^*_2$ do not impose any restrictions on $\al$.

2. Similarly to the classical condition \eqref{1}, the (sub)transversality characterizations in Theorems~\ref{T0}, \ref{T1} and \ref{T2} require that among all admissible (i.e., satisfying all the conditions of the theorems) pairs of nonzero elements $x^*_1$ and $x^*_2$ there is no one with $x^*_1$ and $x^*_2$ oppositely directed.

3. The sum $\|x^*_1\|+\|x^*_2\|$ in Theorems~\ref{T0}, \ref{T1} and \ref{T2} corresponds to the sum norm on $\R^2$, which is dual to the maximum norm on $\R^2$ used in Definitions~\ref{D1} and \ref{D2+}.
It can be replaced by $\max\{\|x^*_1\|,\|x^*_2\|\}$ (cf. \cite[(6.11)]{Pen13}) or any other norm on $\R^2$.

4. Condition \eqref{D1-2} is equivalent to the inequality $\alpha d\left(x,A\cap B\right)\le \max\left\{\|x-a\|,\|x-b\|\right\}$ holding for all triples $x\in\B_{\delta}(\bar{x})$, $a\in A$ and $b\in B$.
Since $\bx\in A\cap B$, it is sufficient to check this inequality only for those triples which satisfy $\alpha \|x-\bx\|>\max\left\{\|x-a\|,\|x-b\|\right\}$.
This simple observation shows that the statement of Theorem~\ref{T1} can be slightly strengthened by adding the following condition: $\norm{x-a}=\norm{x-b}<\alpha\|x-\bx\|$.
\xqed\end{remark}

The proof of Theorem~\ref{T1} follows the sequence proposed in \cite{Kru15} when deducing metric subregularity characterizations for \SVM s and consists of a series of propositions providing lower primal and dual estimates for the constant $\str$ and, thus, sufficient conditions for the subtransversality of the pair $\{A,B\}$ at $\bx$ which can be of independent interest.

First observe that constant $\str$ characterizing subtransversality and introduced in Definition~\ref{D1} can be written explicitly as
\begin{equation}\label{errcon}
\str= \liminf_{\substack{a\to\bar{x},\, b\to\bar x,\, x\to\bar x\\a\in A,\;b\in B,\; x\notin A\cap B}} \frac{f(a,b,x)} {d\left(x,A\cap B\right)}= \liminf_{\substack{a\to\bar{x},\, b\to\bar x,\, x\to\bar x\\x\notin A\cap B}} \frac{\hat f(a,b,x)} {d\left(x,A\cap B\right)},
\end{equation}
with the convention that the infimum over the empty set equals 1, and the functions $f: X^{3}\to\R$ and $\hat f: X^{3}\to\R_\infty$ defined, respectively, by
\begin{gather}\label{f}
f(x_1,x_2,x):=\max\{\norm{x_1-x},\norm{x_2-x}\},\quad x_1,x_2,x\in X,
\\\label{hatf}
\hat f(x_1,x_2,x):= f(x_1,x_2,x)+i_{A\times B}(x_1,x_2),\quad x_1,x_2,x\in X,
\end{gather}
where $i_{A\times B}$ is the indicator function of $A\times B$: $i_{A\times B}(x_1,x_2)=0$ if $x_1\in A$, $x_2\in B$ and $i_{A\times B}(x_1,x_2)=+\infty$ otherwise.

Below,
we are going to use two different norms on $X^3$: a norm depending on a parameter $\rho>0$ and defined as follows:
\begin{equation}\label{norm}
\norm{(x_1,x_2,x)}_{\rho} :=\max\left\{\norm{x},\rho \norm{x_1},\rho \norm{x_2}\right\},\quad x_1,x_2,x\in X,
\end{equation}
and the conventional maximum norm $\norm{(\cdot,\cdot,\cdot)}$ corresponding to $\rho=1$ in the above definition;
we drop the subscript $\rho$ in this case.
It is easy to check that the dual norm corresponding to \eqref{norm} has the following form:
\begin{equation}\label{normd}
\norm{(x^*_1,x^*_2,x^*)}_{\rho} =\|x^*\|+\rho^{-1}(\|x^*_1\|+\|x^*_2\|),\quad x^*_1,x^*_2,x^*\in X^*.
\end{equation}

The next proposition provides an equivalent primal space representation of the subtransversality constant \eqref{errcon}.
Its proof is based on the application of the Ekeland variational principle (Lemma~\ref{Ek}).

\begin{proposition}\label{P4}
Suppose $X$ is a Banach space, $A,B\subset X$ are closed, and $\bx\in A\cap B$.
Then the following representation of the subtransversality constant \eqref{errcon} is true:
\begin{align}\label{str}
\str&= \lim\limits_{\rho\downarrow0} \inf\limits_{\substack{a\in A\cap\B_\rho(\bx),\; b\in B\cap\B_\rho(\bx)\\x\in\B_\rho(\bx),\; \max\{\norm{x-a},\norm{x-b}\}>0}}
\sup_{\substack{a'\in A,\,b'\in B,\, u\in X\\(a',b',u)\neq (a,b,x)}} \frac{\left(f(a,b,x)- f(a',b',u)\right)_+} {\norm{(a',b',u)-(a,b,x)}_{\rho}},
\end{align}
with the convention that the infimum over the empty set equals 1.
\end{proposition}

\begin{proof}
Let $R$ denote the expression in the \RHS\ of \eqref{str}.
We first show that $\str\le R$.
If $\str=0$, the inequality holds trivially.
Let $0<\al<\str$.
By \eqref{errcon}, there is a $\delta>0$ such that
\begin{equation}\label{3ar}
\frac{f(a,b,x)} {d\left(x,A\cap B\right)}>\al
\end{equation}
for all $a\in A\cap\B_\de(\bx)$, $b\in B\cap\B_\de(\bx)$ and $x\in\B_\de(\bx)$ with $x\notin A\cap B$.
Choose a positive $\rho<\min\{\de,(\al+1)\iv\}$ and any $a\in A\cap\B_\rho(\bx)$, $b\in B\cap\B_\rho(\bx)$ and $x\in\B_\rho(\bx)$ with $\max\{\norm{x-a},\norm{x-b}\}>0$.
If $x\notin A\cap B$, then, in view of \eqref{3ar}, one can find a $u\in A\cap B$ such that
$$
\frac{f(a,b,x)} {\norm{u-x}}>\al.
$$
Then,
\begin{align*}
\frac{f(a,b,x)-f(u,u,u)} {\norm{(u,u,u)-(a,b,x)}_{\rho}}
&=\frac{f(a,b,x)} {\max\{\norm{u-x},\rho\norm{u-a},\rho\norm{u-b}\}}
\\
&\ge\frac{f(a,b,x)} {\max\{\norm{u-x},\rho(\norm{u-x}+\norm{x-a}), \rho(\norm{u-x}+\norm{x-b})\}}
\\
&=\frac{f(a,b,x)} {\max\{\norm{u-x},\rho(\max\{\norm{x-a}, \norm{x-b}\}+\norm{u-x})\}}
\\
&=\frac{f(a,b,x)} {\max\{\norm{u-x},\rho(f(a,b,x)+\norm{u-x})\}}
\\
&=\min\left\{\frac{f(a,b,x)}{\norm{u-x}}, \frac{1} {\rho\left(1+\left(\frac{f(a,b,x)} {\norm{u-x}}\right)\iv\right)}\right\}> \al.
\end{align*}
If $x\in A\cap B$, then
\begin{align}\label{22}
\frac{f(a,b,x)-f(x,x,x)} {\norm{(x,x,x)-(a,b,x)}_{\rho}} =\frac{f(a,b,x)} {\max\{\rho\norm{x-a},\rho\norm{x-b}\}} =\rho\iv>\al+1>\al.
\end{align}
Combining the two cases, we obtain
\begin{align*}
\inf\limits_{\substack{a\in A\cap\B_\rho(\bx),\; b\in B\cap\B_\rho(\bx)\\x\in\B_\rho(\bx),\; \max\{\norm{x-a},\norm{x-b}\}>0}}
\sup_{\substack{a'\in A,\,b'\in B,\, u\in X\\(a',b',u)\neq (a,b,x)}} \frac{\left(f(a,b,x)- f(a',b',u)\right)_+} {\norm{(a',b',u)-(a,b,x)}_{\rho}} \ge\al.
\end{align*}
The claimed inequality follows after letting $\rho\downarrow0$ and $\al\uparrow\str$.

Now we show the opposite inequality: $R\le\str$.
Let $\str<\alpha<\infty$.
Choose an $\al'>0$ and a $\rho>0$ such that $\str<\alpha'<\alpha$ and $\rho<1-\al'/\al$, and set
\begin{equation}\label{28}
\eta:= \min\left\{\frac{\rho}{4},\frac{\rho}{2\alpha'}, \rho^{\frac{2}{\rho}}\right\}.
\end{equation}
By \eqref{errcon}, there are $\hat a\in A$, $\hat b\in B$ and $\hat x\in B_{\eta}(\bar x)\setminus(A\cap B)$ such that
\begin{equation}\label{29}
f(\hat a,\hat b,\hat x)<\alpha' d\left(\hat x,A\cap B\right).
\end{equation}
As $\hat x\notin A\cap B$, we have either $\hat x\ne\hat a$ or $\hat x\ne\hat b$; hence $\varepsilon:=f(\hat a,\hat b,\hat x)>0$.
Denote $\mu:=d\left(\hat x,A\cap B\right)$.
Then $0<\eps<\alpha'\mu$ and $\mu\le \norm{\hat x-\bar x}\le \eta\le \frac{\rho}{4}<1$.
Applying to the \lsc\ function \eqref{hatf} the Ekeland variational principle (Lemma~\ref{Ek}) with $\varepsilon$ as above and
\begin{equation}\label{30}
\lambda:=\mu(1-\mu^{\frac{\rho}{2-\rho}})>0,
\end{equation}
we find points $a\in A$, $b\in B$ and $x\in X$ such that
\begin{equation}\label{31}
\|(a,b,x)-(\hat a,\hat b,\hat x)\|_{\rho}<\lambda,\quad f(a,b,x)\le f(\hat a,\hat b,\hat x),
\end{equation}
and
\begin{equation}\label{32}
f(a',b',u)+\frac{\varepsilon}{\lambda} \norm{(a',b',u)-(a,b,x)}_{\rho} \ge f(a,b,x)\qdtx{for all} (a',b',u)\in A\times B\times X.
\end{equation}
Thanks to \eqref{31}, \eqref{30}, \eqref{28} and \eqref{29}, we have
$$
\norm{x-\hat x}<\lambda <\mu \le \norm{\hat x-\bar x},
$$
\begin{equation}\label{33}
d\left(x,A\cap B\right)\ge d\left(\hat x,A\cap B\right)- \norm{x-\hat x}>\mu-\lambda =\mu^{\frac{2}{2-\rho}},
\end{equation}
\begin{equation}\label{34}
\norm{x-\bar x}\le \norm{x-\hat x}+\norm{\hat x-\bar x}<2\norm{\hat x-\bar x}\le 2\eta\le\frac{\rho}{2},
\end{equation}
\begin{equation}\label{35}
f(a,b,x)\le f(\hat a,\hat b,\hat x)<\alpha'\mu\le \alpha'\eta\le\frac{\rho}{2}.
\end{equation}
It follows from \eqref{33} that $x\notin A\cap B$, and consequently, either $x\ne a$ or $x\ne b$.
Besides, by \eqref{34} and \eqref{35},
$$
\norm{x-\bar x}<\rho\qdtx{and} \max\{\norm{a-\bx},\norm{b-\bx}\} \le \max\{\norm{x-a},\norm{x-b}\}+\norm{x-\bar x}<\rho.
$$
Observe that $\mu^{\frac{\rho}{2-\rho}}\le \eta^{\frac{\rho}{2-\rho}}< \eta^{\frac{\rho}{2}}\le \rho$, and consequently, by \eqref{29} and \eqref{30},
$$
\frac{\varepsilon}{\lambda} <\frac{\alpha'\mu}{\lambda} =\frac{\alpha'}{1-\mu^{\frac{\rho}{2-\rho}}} <\frac{\alpha'}{1-\rho}<\alpha.
$$
Thanks to \eqref{32} and \eqref{f}, we have
$$
f(a,b,x)-f(a',b',u)\le \al\norm{(a',b',u)-(a,b,x)}_{\rho} \qdtx{for all} (a',b',u)\in A\times B\times X.
$$
It follows that
\begin{align*}
\inf\limits_{\substack{a\in A\cap\B_\rho(\bx),\; b\in B\cap\B_\rho(\bx)\\x\in\B_\rho(\bx),\; \max\{\norm{x-a},\norm{x-b}\}>0}}
\sup_{\substack{a'\in A,\,b'\in B,\, u\in X\\(a',b',u)\neq (a,b,x)}} \frac{f(a,b,x)- f(a',b',u)} {\norm{(a',b',u)-(a,b,x)}_{\rho}} \le\al.
\end{align*}
Taking limits in the last inequality as $\rho\downarrow 0$ and $\alpha\downarrow\str$ yields the claimed inequality.
\qed\end{proof}

\begin{remark}\label{R3}
1. The \RHS\ of \eqref{str} is the \emph{uniform strict outer slope} \cite{Kru15} of the function \eqref{hatf} (considered as a function of two variables $x$ and $(x_1,x_2)$) at $(\bx,(\bx,\bx))$.

2. The inequality `$\le$' in \eqref{str} is valid in arbitrary (not necessarily complete) normed linear spaces.
The completeness of the space $X$ is only needed for the inequality `$\ge$', the proof of which is based on the application of the Ekeland variational principle.
\xqed\end{remark}

The next proposition provides another two primal space  representations of the subtransversality constant \eqref{errcon} which impose additional restrictions on the choice of $a$, $b$ and $x$ under the $\inf$ in \eqref{str}.

\begin{proposition}\label{P5}
Suppose $X$ is a Banach space, $A,B\subset X$ are closed, and $\bx\in A\cap B$.
Then the following representations of the subtransversality constant \eqref{errcon} are true:
\begin{align}\notag
\str&=\lim\limits_{\rho\downarrow0} \inf\limits_{\substack{a\in(A\setminus B)\cap\B_\rho(\bx),\;b\in(B\setminus A)\cap\B_\rho(\bx)\\x\in\B_\rho(\bx)}}
\sup_{\substack{a'\in A,\,b'\in B,\, u\in X\\(a',b',u)\neq (a,b,x)}} \frac{\left(f(a,b,x)- f(a',b',u)\right)_+} {\norm{(a',b',u)-(a,b,x)}_{\rho}}
\\\label{str+}
&=\lim\limits_{\rho\downarrow0} \inf\limits_{\substack{a\in(A\setminus B)\cap\B_\rho(\bx),\;b\in(B\setminus A)\cap\B_\rho(\bx)\\x\in\B_\rho(\bx),\; \norm{x-a}=\norm{x-b}}}
\sup_{\substack{a'\in A,\,b'\in B,\, u\in X\\(a',b',u)\neq (a,b,x)}} \frac{\left(f(a,b,x)- f(a',b',u)\right)_+} {\norm{(a',b',u)-(a,b,x)}_{\rho}},
\end{align}
with the convention that the infimum over the empty set equals 1.
\end{proposition}

\begin{proof}
Let $R$, $R_1$ and $R_2$ denote the \RHS\ of \eqref{str}, and the first and the second expressions in \eqref{str+}, respectively.
Comparing the sets of restrictions on the choice of $a$, $b$ and $x$ under the $\inf$ in these expressions, it is easy to observe that $R\le R_1\le R_2$.
Next we show that both inequalities hold as equalities.

\underline{$R=R_1$}.
Let $a\in A$, $b\in B$, $x\in X$,
$\max\{\norm{x-a},\norm{x-b}\}>0$, and $\rho\in]0,1/2[$.
If $b\in A$, i.e., $b\in A\cap B$, then $f(b,b,b)=0$
and
\begin{align*}
\norm{(b,b,b)-(a,b,x)}_{\rho} &=\max\left\{\norm{b-x},\rho \norm{b-a}\right\} \le\max\left\{\norm{b-x},\rho (\norm{b-x}+\norm{a-x})\right\}
\\
&\le\max\left\{1,2\rho\right\} \max\left\{\norm{b-x},\norm{a-x}\right\} =\max\left\{\norm{b-x},\norm{a-x}\right\}.
\end{align*}
Similarly, if $a\in B$, i.e., $a\in A\cap B$, then $f(a,a,a)=0$ and
\begin{align*}
\norm{(a,a,a)-(a,b,x)}_{\rho} \le\max\left\{\norm{b-x},\norm{a-x}\right\}.
\end{align*}
Thus, in both cases,
\begin{equation}\label{P5P-1}
\sup_{\substack{u\in X,\,a'\in A,\,b'\in B\\(a',b',u)\neq (a,b,x)}} \frac{f(a,b,x)- f(a',b',u)} {\norm{(a',b',u)-(a,b,x)}_{\rho}}\ge1.
\end{equation}
Since $\str\le1$, all points $a$ and $b$ with either $a\in A\cap B$ or $b\in A\cap B$
can be excluded when computing $\str$ using \eqref{str}.
This proves $R=R_1$.

\underline{$R_1=R_2$}.
Let $a\in A$, $b\in B$, $x\in X$ and $\rho>0$.
If $\norm{x-a}<\norm{x-b}$, then $f(a,b,x)=\norm{x-b}$.
Taking $u_t:=x-t(x-b)$ for $t>0$, we have
$f(a,b,u_t)=(1-t)\norm{x-b}$ for all sufficiently small $t>0$, and $\norm{(a,b,u_t)-(a,b,x)}_{\rho} =\norm{u_t-x}=t\norm{x-b}$.
Hence,
\begin{equation}\label{P5P-2}
\frac{f(a,b,x)- f(a,b,u_t)} {\norm{(a,b,u_t)-(a,b,x)}_{\rho}}=1.
\end{equation}
Similarly, if $\norm{x-b}<\norm{x-a}$, then we can take $u_t:=x-t(x-a)$ to arrive at the same equality \eqref{P5P-2} for all sufficiently small $t>0$.
Thus, in both cases,
inequality \eqref{P5P-1} holds, and points with $\norm{x-a}\ne\norm{x-b}$
can be excluded when computing $\str$ using the first representation in \eqref{str+}.
\qed\end{proof}

\begin{remark}
The expression after $\sup$ in the \RHS s of \eqref{str} and \eqref{str+} can be greater than 1 (see \eqref{22} when $\rho<1$).
Nevertheless, $\str$ computed in accordance with \eqref{str} or \eqref{str+} (under the conventions employed in Propositions~\ref{P4} and \ref{P5}) is always less than or equal to 1.
\xqed\end{remark}

Now we define a `localized' subtransversality constant:
\begin{gather}\label{str1}
{\rm str}_1[A,B](\bar x):= \lim\limits_{\rho\downarrow0} \inf\limits_{\substack{a\in(A\setminus B)\cap\B_\rho(\bx),\;b\in(B\setminus A)\cap\B_\rho(\bx)\\x\in\B_\rho(\bx)}}
\limsup_{\substack{a'\to a,\,b'\to b,\, u\to x\\ a'\in A,\,b'\in B\\(a',b',u)\neq (a,b,x)}} \frac{\left(f(a,b,x)- f(a',b',u)\right)_+} {\norm{(a',b',u)-(a,b,x)}_{\rho}},
\end{gather}
with the convention that the infimum over the empty set equals 1.
It corresponds to the first expression in \eqref{str+} with $\sup$ replaced by $\limsup$.
Observe that
\begin{equation*}
\limsup_{\substack{u\to x,\,a'\to a,\,b'\to b\\ a'\in A,\,b'\in B\\(a',b',u)\neq (a,b,x)}} \frac{\left(f(a,b,x)- f(a',b',u)\right)_+} {\norm{(a',b',u)-(a,b,x)}_{\rho}}.
\end{equation*}
in the above definition
is the \emph{$\rho$-slope} \cite{Kru15} (i.e., the 
slope \cite{DMT,Aze03,Iof00_,FabHenKruOut10} with respect to the distance in $ X^3$ corresponding to the norm defined by \eqref{norm}) at $(x,(a,b))$ of the function $(u,(a',b'))\mapsto f(a',b',u)$.

\begin{proposition}\label{P6}
Suppose $X$ is a normed linear space, $A,B\subset X$ are closed, and $\bx\in A\cap B$.
Then the following representation of the subtransversality constant \eqref{str1} is true:
\begin{gather}\label{str1"}
{\rm str}_1[A,B](\bar x)= \lim\limits_{\rho\downarrow0} \inf\limits_{\substack{a\in(A\setminus B)\cap\B_\rho(\bx),\;b\in(B\setminus A)\cap\B_\rho(\bx)\\x\in\B_\rho(\bx),\; \norm{x-a}=\norm{x-b}}}
\limsup_{\substack{a'\to a,\,b'\to b,\, u\to x\\ a'\in A,\,b'\in B\\(a',b',u)\neq (a,b,x)}} \frac{\left(f(a,b,x)- f(a',b',u)\right)_+} {\norm{(a',b',u)-(a,b,x)}_{\rho}},
\end{gather}
with the convention that the infimum over the empty set equals 1.
\end{proposition}

\begin{proof}
The proof follows that of the last equality in \eqref{str+}.
Comparing \eqref{str1} and \eqref{str1"}, we immediately get the inequality ${\rm str}_1[A,B](\bar x)\le R$, where $R$ denotes the \RHS\ of \eqref{str1"}.
If $a\in A$, $b\in B$, $x\in X$, and $\norm{x-a}<\norm{x-b}$, we take $u_t:=x-t(x-b)$ for $t>0$ and arrive at the equality \eqref{P5P-2} valid for all sufficiently small $t>0$.
This yields an analogue of the inequality \eqref{P5P-1} with $\sup$ replaced by $\limsup$ as in \eqref{str1}.
The same argument applies in the case $\norm{x-b}<\norm{x-a}$.
As a result, the points with $\norm{x-b}\ne\norm{x-a}$
can be excluded when computing ${\rm str}_1[A,B](\bar x)$ using definition \eqref{str1}.
This proves representation \eqref{str1"}.
\qed\end{proof}

\begin{remark}
One can define an analogue of ${\rm str}_1[A,B](\bar x)$ using the limiting procedure in the representation of $\str$ in \eqref{str}.
Unlike the `nonlocal' case in Propositions~\ref{P4} and \ref{P5}, such an analogue does not coincide in general with ${\rm str}_1[A,B](\bar x)$ defined by \eqref{str1}, although it can still be used for formulating sufficient conditions of subtransversality.
In this paper, we are not going to use quantities defined with the help of the limiting procedure in the representation of $\str$ in \eqref{str}.
\xqed\end{remark}

The next proposition clarifies the relationship between ${\rm str}_1[A,B](\bar x)$ and $\str$.

\begin{proposition}\label{P7}
Suppose $X$ is a Banach space, $A,B\subset X$ are closed, and $\bx\in A\cap B$.
Then
\begin{enumerate}
\item
${\rm str}_1[A,B](\bar x)\le\str$;
\item
if $A$ and $B$ are convex, then {\rm (i)} holds as equality.
\end{enumerate}
\end{proposition}

\begin{proof}
(i) is an immediate consequence of the definition \eqref{str1} and the first representation in \eqref{str+} (or Proposition~\eqref{P6} and the second representation in \eqref{str+}).

(ii) Let $A$ and $B$ be convex.
Then function $f$ defined by \eqref{f} is convex.
For any $a\in A\setminus B$, $b\in B\setminus A$, $x\in X$, we have $f(a,b,x)>0$.
Hence, for any $\rho>0$, and any $a'\in A$, $b'\in B$ and $u\in X$ with $f(a',b',u)<f(a,b,x)$ (such a triple exists, e.g., $f(\bx,\bx,\bx)=0$), we have
\begin{align*}
\frac{f(a,b,x)- f(a',b',u)} {\norm{(a,b,x)-(a',b',u)}_{\rho}} &\le\lim_{t\downarrow0}\frac{f(a,b,x)- f((a,b,x)+t((a',b',u)-(a,b,x))} {\norm{(a,b,x)-((a,b,x)+t((a',b',u)-(a,b,x))}_{\rho}}
\\
&\le\limsup_{\substack{u\to x,\,a'\to a,\,b'\to b\\ a'\in A,\,b'\in B\\(a',b',u)\neq (a,b,x)}} \frac{\left(f(a,b,x)- f(a',b',u)\right)_+} {\norm{(a',b',u)-(a,b,x)}_{\rho}}.
\end{align*}
In view of the first representation in \eqref{str+} and definition \eqref{str1}, we have $\str\le{\rm str}_1[A,B](\bar x)$.
In view of (i), this proves (ii).
\qed\end{proof}

\begin{remark}
Proposition~\ref{P7} is valid in arbitrary (not necessarily complete) normed linear spaces if $\str$ is defined by
one of the expressions in \eqref{str+} (see Remark~\ref{R3}.2).
\sloppy
\xqed\end{remark}

To proceed to dual characterizations of subtransversality, we need a representation of the subdifferential of the convex function $f$ given by \eqref{f}.
It is computed in the next lemma which improves (in the current setting) \cite[Lemma~4.2]{KruTha15}.

\begin{lemma}\label{l02}
Let $X$ be a normed space and
$f$ be given by \eqref{f}.
Then
\begin{gather}\label{l02a-1}
\sd f(x_1,x_2,x) =\left\{(x_1^*,x_2^*,-x_{1}^*-x_{2}^*)\in (X^*)^3\mid (x_1^*,x_2^*)\in \sd g(x_1-x,x_2-x)\right\},\quad x_1,x_2,x\in X,
\end{gather}
where $g$ is the maximum norm on $X^2$:
\begin{gather}\label{l02a-2}
g(x_1,x_2):=\max\{\norm{x_1},\norm{x_2}\},\quad x_1,x_2\in X.
\end{gather}
If $x_1\ne x$ or $x_2\ne x$, then $(x_1^*,x_2^*,x^*)\in \sd f(x_1,x_2,x)$ if and only if the following conditions are satisfied:
\begin{gather*}
x_{1}^*+x_{2}^*+x^*=0,
\quad
\|x_1^*\|+\|x_2^*\|=1,
\\
\langle x_1^*,x_1-x\rangle= \|x_1^*\|\norm{x_1-x},
\quad
\langle x_2^*,x_2-x\rangle= \|x_2^*\|\norm{x_2-x},
\\
\mbox{if}\;
\norm{x_1-x}<\norm{x_2-x},
\quad\mbox{then}\quad
x_1^*=0,
\\
\mbox{if}\;
\norm{x_2-x}<\norm{x_1-x},
\quad\mbox{then}\quad
x_2^*=0.
\end{gather*}
\end{lemma}
\begin{proof}
The convex function $f$ given by \eqref{f} is a composition of the continuous linear mapping
\begin{gather}\label{l02ap-1}
(x_1,x_2,x)\mapsto(x_1-x,x_2-x)
\end{gather}
from $X^3$ to $X^2$ and the norm \eqref{l02a-2} on $X^2$.
The mapping adjoint to \eqref{l02ap-1} is from $(X^*)^2$ to $(X^*)^3$ and is of the form
\begin{gather*}
(x_1^*,x_2^*)\mapsto(x_1^*,x_2^*,-x_{1}^*-x_{2}^*).
\end{gather*}
Representation \eqref{l02a-1} is a consequence of the standard convex chain rule (cf., e.g., \cite[Theorem~4.2.2]{IofTik79}).

The dual norm corresponding to \eqref{l02a-2} is of the form $(x_1^*,x_2^*)\mapsto\|x_1^*\|+\|x_2^*\|$.
Hence (cf., e.g., \cite[Subection~0.3.2]{IofTik79}, \cite[Corollary~2.4.16]{Zal02}), if $(x_1,x_2)\ne0$, then
\begin{gather}\label{l02ap-2}
\sd g(x_1,x_2)=\left\{(x_1^*,x_2^*)\in (X^*)^2\mid \|x_1^*\|+\|x_2^*\|=1,\; \ang{(x_1^*,x_2^*),(x_1,x_2)} =\max\{\norm{x_1},\norm{x_2}\}\right\}.
\end{gather}
If $\|x_1^*\|+\|x_2^*\|=1$, then the last condition in \eqref{l02ap-2} is equivalent to the following group of conditions:
\begin{gather*}
\langle x_1^*,x_1\rangle= \|x_1^*\|\norm{x_1},
\quad
\langle x_2^*,x_2\rangle= \|x_2^*\|\norm{x_2},
\\
\mbox{if}\;
\norm{x_1}<\norm{x_2},
\quad\mbox{then}\quad
x_1^*=0,
\\
\mbox{if}\;
\norm{x_2}<\norm{x_1},
\quad\mbox{then}\quad
x_2^*=0.
\end{gather*}
The second part of the proposition follows now from the representation \eqref{l02a-1}.
\qed\end{proof}

The subtransversality constant \eqref{str1} admits dual estimates which are crucial for the conclusions of Theorems~\ref{T1} and \ref{T2}.
In what follows we will use notations $\itrd{w}$ and $\itrd{c}$ for the supremum of all $\al$ in Theorems~\ref{T1} and \ref{T2}, respectively, with the convention that the supremum over the empty set equals 0.
It is easy to check the following explicit representations of the two constants:

\begin{align}\notag
\itrd{w}:=& \lim_{\rho\downarrow0} \inf_{\substack{a\in(A\setminus B)\cap\B_\rho(\bx),\;b\in(B\setminus A)\cap\B_\rho(\bx)\\x\in\B_\rho(\bx),\; \norm{x-a}=\norm{x-b}}}
\\\label{itrw}
&\liminf_{\substack{x'\to x,\;x_1'\to a,\;x_2'\to b,\; a'\to a,\;b'\to b\\a'\in A,\;b'\in B,\; \|x'-x_1'\|=\|x'-x_2'\|\\ d(x_1^*,N_{A}(a'))<\rho,\; d(x_2^*,N_{B}(b'))<\rho,\; \|x_1^*\|+\|x_2^*\|=1
\\
\langle x_1^*,x'-x_1'\rangle=\|x_1^*\|\,\|x'-x_1'\|,\;
\langle x_2^*,x'-x_2'\rangle=\|x_2^*\|\,\|x'-x_2'\|}}
\|x^*_1+x^*_2\|,
\\\label{itrc}
\itrd{c}:=& \liminf_{\substack{x\to\bx,\;a\to\bx,\;b\to\bx\\a\in A\setminus B,\;b\in B\setminus A,\; \norm{x-a}=\norm{x-b}\\ d(x_1^*,N_{A}(a))\to0,\;d(x_2^*,N_{B}(b))\to0,\; \|x_1^*\|+\|x_2^*\|=1
\\
\langle x_1^*,x-a\rangle=\|x_1^*\|\,\|x-a\|,\;
\langle x_2^*,x-b\rangle=\|x_2^*\|\,\|x-b\|}}
\|x^*_1+x^*_2\|,
\end{align}
with the convention that the infimum over the empty set equals 1.

\begin{proposition}\label{P8}
Suppose $X$ is a Banach space, $A,B\subset X$ are closed, and $\bx\in A\cap B$.
\begin{enumerate}
\item
If either $X$ is Asplund or $A$ and $B$ are convex, then the following dual representations of the subtransversality constant \eqref{str1} are true:
\begin{align}\notag
{\rm str}_1[A,B](\bar x)&= \lim\limits_{\rho\downarrow0} \inf\limits_{\substack{a\in(A\setminus B)\cap\B_\rho(\bx),\;b\in(B\setminus A)\cap\B_\rho(\bx)\\x\in\B_\rho(\bx)\\ (x^*_1,x^*_2,x^*)\in\partial\hat f(a,b,x),\; \norm{x^*_1}+\norm{x^*_2}<\rho}}
\norm{x^*}
\\\label{str1'}
&= \lim\limits_{\rho\downarrow0} \inf\limits_{\substack{a\in(A\setminus B)\cap\B_\rho(\bx),\;b\in(B\setminus A)\cap\B_\rho(\bx)\\x\in\B_\rho(\bx),\; \norm{x-a}=\norm{x-b}\\ (x^*_1,x^*_2,x^*)\in\partial\hat f(a,b,x),\; \norm{x^*_1}+\norm{x^*_2}<\rho}}
\norm{x^*},
\end{align}
where the function $\hat f:X^3\to\R_\infty$ is defined by \eqref{hatf}
and the convention that the infimum over the empty set equals 1 is in force.
Moreover,
\item
if $X$ is Asplund, then
${\rm str}_1[A,B](\bar x)\ge\itrd{w}$;
\item
if $A$ and $B$ are convex, then ${\rm str}_1[A,B](\bar x)=\itrd{c}$.
\end{enumerate}
\end{proposition}

\begin{proof}
(i) Let $R_1$ and $R_2$ denote the first and second expressions in \eqref{str1'}, respectively.
We first show that
${\rm str}_1[A,B](\bar x)\le R_1$.
Let $\rho>0$, $a\in A$, $b\in B$, $x\in X$, $(x^*_1,x^*_2,x^*)\in\partial\hat f(a,b,x)$ and $\norm{x^*_1}+\norm{x^*_2}<\rho^2$.
Then, using the definition \eqref{sd} of the Fr\'echet subdifferential and representation \eqref{normd} of the dual norm, we obtain
\begin{align*}
\limsup_{\substack{a'\to a,\,b'\to b,\, u\to x\\ a'\in A,\,b'\in B\\(a',b',u)\neq (a,b,x)}} &\frac{f(a,b,x)- f(a',b',u)} {\norm{(a',b',u)-(a,b,x)}_{\rho}} \le\norm{(x^*_1,x^*_2,x^*)}_\rho
\\
&-\liminf_{\substack{a'\to a,\,b'\to b,\, u\to x\\ a'\in A,\,b'\in B\\(a',b',u)\neq (a,b,x)}} \frac{f(a',b',u)- f(a,b,x)-\ang{(x^*_1,x^*_2,x^*),(a',b',u)-(a,b,x)}} {\norm{(a',b',u)-(a,b,x)}_{\rho}}
\\
&\le\norm{(x^*_1,x^*_2,x^*)}_\rho =\|x^*\|+\rho^{-1}(\|x^*_1\|+\|x^*_2\|)\le\norm{x^*}+\rho.
\end{align*}
If $\rho<1$, then $\rho^2<\rho$ and it follows from the above estimate that
\begin{multline*}
\inf\limits_{\substack{a\in(A\setminus B)\cap\B_\rho(\bx),\;b\in(B\setminus A)\cap\B_\rho(\bx)\\x\in\B_\rho(\bx)}}
\limsup_{\substack{u\to x,\,a'\to a,\,b'\to b\\ a'\in A,\,b'\in B\\(a',b',u)\neq (a,b,x)}} \frac{f(a,b,x)- f(a',b',u)} {\norm{(a',b',u)-(a,b,x)}_{\rho}} \\
\le\inf\limits_{\substack{a\in(A\setminus B)\cap\B_{\rho^2}(\bx),\;b\in(B\setminus A)\cap\B_{\rho^2}(\bx)\\x\in\B_{\rho^2}(\bx)\\\\ (x^*_1,x^*_2,x^*)\in\partial\hat f(a,b,x),\; \norm{x^*_1}+\norm{x^*_2}<\rho^2}}
\norm{x^*}+\rho.
\end{multline*}
Passing to the limits as $\rho\downarrow0$ and using definition \eqref{str1}, we arrive at the inequality
${\rm str}_1[A,B](\bar x)\le R_1$.

Next we prove the opposite inequality.
Let ${\rm str}_1[A,B](\bar x)<\beta<\alpha<\infty$, ${\rho>0}$ and
${\rho':=\min\{1,\alpha^{-1}\}\rho}$.
By \eqref{str1"}, one can find points $\hat a\in(A\setminus B)\cap\B_{\rho'}(\bx)$, $\hat b\in(B\setminus A)\cap\B_{\rho'}(\bx)$ and $\hat x\in\B_{\rho'}(\bx)$, such that $\|\hat x-\hat a\|=\|\hat x-\hat b\|$ and
$$
f(\hat a,\hat b,\hat x)- f(a',b',u) \le\be\|(a',b',u)-(\hat a,\hat b,\hat x)\|_{\rho'} \qdtx{for all} (a',b',u)\in A\times B\times X\;\;\mbox{near}\;\;(\hat a,\hat b,\hat x).
$$
In other words, $(\hat a,\hat b,\hat x)$ is a local minimizer of the function
$$
(a',b',u)\mapsto\hat f(a',b',u)+\beta\|(a',b',u)-(\hat a,\hat b,\hat x)\|_{\rho'},
$$
and consequently, its Fr\'echet subdifferential at $(\hat a,\hat b,\hat x)$ contains zero.
We consider two cases.

1) $X$ is an Asplund space.
Take an $\eps>0$ such that
$$\eps<\min\{d(\hat a,B),d(\hat b,A)\},\quad \norm{\hat x-\bar x}+\eps<\rho',\quad \norm{\hat a-\bx}+\eps<\rho', \quad \|\hat b-\bx\|+\eps<\rho', \quad
\be+\eps<\al.$$
Applying the fuzzy sum rule for Fr\'echet subdifferentials (Lemma~\ref{SR}(i))
and the representation \eqref{normd} of the dual norm,
we can find points $a\in A\cap\B_\eps(\hat a)$, $b\in B\cap\B_\eps(\hat b)$, $x\in\B_\eps(\hat x)$ and $(x^*_1,x^*_2,x^*)\in \partial\hat f(a,b,x)$ such that \begin{gather*}
\norm{(x^*_1,x^*_2,x^*)}_{\rho'} =\|x^*\|+(\|x^*_1\|+\|x^*_2\|)/\rho' <\beta+\varepsilon.
\end{gather*}
It follows that $a\in(A\setminus B)\cap\B_\rho(\bx)$, $b\in(B\setminus A)\cap\B_\rho(\bx)$ and $x\in\B_\rho(\bx)$.

2) $A$ and $B$ are convex.
Then function $\hat f$ is convex.
Applying the convex sum rule (Lemma~\ref{SR}(ii)), we can find a subgradient $(x^*_1,x^*_2,x^*)\in \partial\hat f(\hat a,\hat b,\hat x)$ such that
\begin{gather*}
\norm{(x^*_1,x^*_2,x^*)}_{\rho'} =\|x^*\|+(\|x^*_1\|+\|x^*_2\|)/\rho'\le\beta.
\end{gather*}

Thus, in both cases we have
\begin{gather*}
\|x^*\|+(\|x^*_1\|+\|x^*_2\|)/\rho'<\al,
\end{gather*}
and consequently,
$$\norm{x^*}<\alpha,\quad \norm{x^*_1}<\rho\qdtx{and} \norm{x^*_2}<\rho.$$
It follows that
$R_1\le\alpha$.
By letting $\alpha\to{\rm str}_1[A,B](\bar x)$, we obtain the claimed inequality.

Observe that, unlike the first case, in the second one we did not produce a new triple $(a,b,x)$ to replace $(\hat a,\hat b,\hat x)$, so the equality $\|\hat x-\hat a\|=\|\hat x-\hat b\|$ is preserved.
Hence, in the convex case both representations in \eqref{str1'} have been proved.

Now we proceed to the proof of the `moreover' part of the proposition.
The remaining equality $R_1=R_2$ in the Asplund space case will be established in the process.

(ii) Suppose $X$ is Asplund.
Let ${\rm str}_1[A,B](\bar x)<\al<1$
and $\rho>0$.
By the first representation in \eqref{str1'} proved above, there are
$a\in(A\setminus B)\cap\B_\rho(\bx)$, $b\in(B\setminus A)\cap\B_\rho(\bx)$, $x\in\B_\rho(\bx)$ and $(w^*_1,w^*_2,w^*) \in\partial\hat f(a,b,x)$, where $\hat f$ is given by \eqref{hatf}, such that
\begin{gather}\label{P8P-45}
\|w^*_1\|+\|w^*_2\|<\rho\qdtx{and} \|w^*\|<\al.
\end{gather}
Denote $\de_0:=\max\{\norm{x-a},\norm{x-b}\}>0$, $\de_1:=\big|\norm{x-a}-\norm{x-b}\big|$ and choose an $\eps>0$ such that
\begin{gather}\label{P8P-21}
\eps<\min\left\{d(a,B),d(b,A),\frac{\de_0}{2}\right\},
\\\label{P8P-22}
\mbox{if}\; \de_1>0
\quad\mbox{then}\quad
\varepsilon<\frac{\de_1}{4},
\\\label{P8P-23}
\|x-\bar x\|+\varepsilon<\rho,\quad \de_0+2\varepsilon<\rho,
\\\label{P8P-8}
\|w^*_1\|+\varepsilon<\rho,\quad
\|w^*_2\|+\varepsilon<\rho,\quad
\|w^*\|+\varepsilon<\al.
\end{gather}

Observe that function $\hat f$ is the sum of two functions: the Lipschitz continuous function $f$ defined by \eqref{f} and the lower semicontinuous indicator function $i_{A\times B}$ (considered as a function on $ X^{3}$).
We can apply the fuzzy sum rule for Fr\'echet subdifferentials (Lemma~\ref{SR}(i)):
there exist points $x_1',x_2',x'\in X$, $a'\in A$, $b'\in B$, $x_1^*,x_2^*,x^*,u_1^*,u_2^*\in X^*$ such that
\begin{gather}\label{P8P-11}
\|x'-x\|<\eps,\quad \|x_1'-a\|<\eps,\quad \|x_2'-b\|<\eps, \quad\|a'-a\|<\eps,\quad \|b'-b\|<\eps,
\\\label{P8P-12}
\left(-x_1^*,-x_2^*,x^*\right)\in\partial f(x_1',x_2',x'),\quad u_1^*\in N_{A}(a'),\quad u_2^*\in N_{B}(b'),
\\\notag
\|(w^*_1,w^*_2,w^*)-(-x_1^*,-x_2^*,x^*) -(u_1^*,u_2^*,0)\|<\varepsilon.
\end{gather}
The last inequality is equivalent to the following three:
\begin{gather}\label{P8P-6}
\|w^*-x^*\|<\varepsilon,\quad
\|w^*_1+x_1^*-u_1^*\|<\varepsilon,\quad
\|w^*_2+x_2^*-u_2^*\|<\varepsilon.
\end{gather}
Thanks to \eqref{P8P-21}, \eqref{P8P-11} and \eqref{P8P-23}, we have $a'\notin B$, $b'\notin A$ and the following estimates:
\begin{align}\notag
\max\{\norm{x_1'-x'},\norm{x_2'-x'}\}\ge \max\{&\norm{x-a}-\|x_1'-a\|-\|x'-x\|,
\\\label{P8P-7}
&\norm{x-b}-\|x_2'-b\|-\|x'-x\|\}
>\de_0-2\eps>0.
\end{align}
If $\de_1>0$ then, in view of \eqref{P8P-22} and \eqref{P8P-11},
\begin{gather}\label{P8P-9}
\big|\norm{x_1'-x'}-\norm{x_2'-x'}\big|\ge \big|\norm{a-x}-\norm{b-x}\big|- \norm{x_1'-a}-\norm{x_2'-b}-2\norm{x'-x}>\de_1-4\eps>0.
\end{gather}

Thanks to \eqref{P8P-7} and Lemma~\ref{l02},
we have
\begin{gather}\label{P8P-1}
x^*=x_1^*+x_2^*,
\\\label{P8P-2}
\|x_1^*\|+\|x_2^*\|=1,
\\\label{P8P-3}
\langle x_1^*,x'-x_1'\rangle=\|x_1^*\|\,\|x'-x_1'\|,
\quad
\langle x_2^*,x'-x_2'\rangle=\|x_2^*\|\,\|x'-x_2'\|,
\\\label{P8P-4}
\mbox{if}\;
\norm{x_1'-x'}<\norm{x_2'-x'},
\quad\mbox{then}\quad
x_1^*=0,
\\\label{P8P-5}
\mbox{if}\;
\norm{x_2'-x'}<\norm{x_1'-x'},
\quad\mbox{then}\quad
x_2^*=0.
\end{gather}
It follows from \eqref{P8P-1}, the first inequality in \eqref{P8P-6} and the second inequality in \eqref{P8P-8} that
\begin{gather}\label{P8P-25}
\|x_1^*+x_2^*\|=\|x^*\|\le\|w^*\|+\varepsilon<\al.
\end{gather}
Then $\|x_2^*\|-\|x_1^*\|<\al$, $\|x_1^*\|-\|x_2^*\|<\al$ and, in view of \eqref{P8P-2},
\begin{gather}\label{P8P-27}
\|x_1^*\|>\frac{1-\al}{2}>0
\qdtx{and}
\|x_2^*\|>\frac{1-\al}{2}>0.
\end{gather}
Hence, by \eqref{P8P-4}, \eqref{P8P-5}, and \eqref{P8P-9}, we have $\norm{x_1'-x'}=\norm{x_2'-x'}$ and $\de_1=0$, i.e., $\de_0=\norm{x-a}=\norm{x-b}$.
This proves the second equality in \eqref{str1'}.
Inequalities \eqref{P8P-6} and \eqref{P8P-8} yield the following estimates:
\begin{gather*}
d(x_1^*,N_{A}(a')) \le\|x_1^*-u_1^*\| <\|w^*_1\|+\varepsilon<\rho,
\\
d(x_2^*,N_{B}(b')) \le\|x_2^*-u_2^*\| <\|w^*_2\|+\varepsilon<\rho.
\end{gather*}
In view of
\eqref{P8P-2}, \eqref{P8P-3} and \eqref{P8P-25}, after taking limits as $\eps\downarrow0$,
we conclude that
\begin{align*}
\liminf_{\substack{x'\to x,\;x_1'\to a,\;x_2'\to b,\; a'\to a,\;b'\to b\\a'\in A,\;b'\in B,\; \|x'-x_1'\|=\|x'-x_2'\|\\ d(x_1^*,N_{A}(a'))<\rho,\; d(x_2^*,N_{B}(b'))<\rho,\; \|x_1^*\|+\|x_2^*\|=1
\\
\langle x_1^*,x'-x_1'\rangle=\|x_1^*\|\,\|x'-x_1'\|,\;
\langle x_2^*,x'-x_2'\rangle=\|x_2^*\|\,\|x'-x_2'\|}}
\|x^*_1+x^*_2\|\le\alpha.
\end{align*}
By letting $\rho\downarrow0$ and $\alpha\downarrow{\rm str}_1[A,B](\bar x)$, we obtain the claimed inequality.

(iii) Let $A$ and $B$ be convex.
We first prove the inequality $\itrd{c}\le {\rm str}_1[A,B](\bar x)$ by modifying slightly (simplifying!) the above proof of (i) replacing the fuzzy sum rule for Fr\'echet subdifferentials by the exact convex sum rule.

Let ${\rm str}_1[A,B](\bar x)<\al<1$
and $\rho>0$.
By the second representation in
\eqref{str1'} proved above, there are
$a\in(A\setminus B)\cap\B_\rho(\bx)$, $b\in(B\setminus A)\cap\B_\rho(\bx)$, $x\in\B_\rho(\bx)$ with $\norm{x-a}=\norm{x-b}$, and $(w^*_1,w^*_2,w^*)\in\partial\hat f(a,b,x)$, where $\hat f$ is given by \eqref{hatf}, satisfying
conditions \eqref{P8P-45}.
Observe that function $\hat f$ is the sum of two convex functions: the Lipschitz continuous function $f$ defined by \eqref{f} and the indicator function $i_{A\times B}$ (considered as a function on $ X^{3}$).
We can apply the convex sum rule (Lemma~\ref{SR}(ii)): there exist a subgradient $(-x_1^*,-x_2^*,x^*)\in\partial f(a,b,x)$ and normals $u_1^*\in N_{A}(a)$ and $u_2^*\in N_{B}(b)$ such that
\begin{gather}\label{P8P-6'}
w^*=x^*,\quad
w^*_1=u_1^*-x_1^*,\quad
w^*_2=u_2^*-x_2^*.
\end{gather}
Thanks to Lemma~\ref{l02}, conditions \eqref{P8P-1} and \eqref{P8P-2} hold true as well as the following two:
\begin{gather}\label{P8P-3'}
\langle x_1^*,x-a\rangle=\|x_1^*\|\,\|x-a\|,
\quad
\langle x_2^*,x-b\rangle=\|x_2^*\|\,\|x-b\|.
\end{gather}
It follows from \eqref{P8P-1}, \eqref{P8P-45}
and the first equality in \eqref{P8P-6'} that
\begin{gather}\label{P8P-25'}
\|x_1^*+x_2^*\|=\|x^*\|=\|w^*\|<\al.
\end{gather}
Then $\|x_2^*\|-\|x_1^*\|<\al$, $\|x_1^*\|-\|x_2^*\|<\al$ and, in view of \eqref{P8P-2}, inequalities \eqref{P8P-27} hold true.
Conditions \eqref{P8P-45}
and \eqref{P8P-6'} yield the following estimates:
\begin{gather*}
d(x_1^*,N_{A}(a)) \le\|x_1^*-u_1^*\| =\|w^*_1\|<\rho,
\\
d(x_2^*,N_{B}(b)) \le\|x_2^*-u_2^*\| =\|w^*_2\|<\rho.
\end{gather*}
Hence, $\itrd{c}\le\alpha$.
By letting $\alpha\downarrow{\rm str}_1[A,B](\bar x)$, we obtain the claimed inequality.

Let $\itrd{c}<\al<1$
and $\rho\in]0,1[$.
By definition \eqref{itrc}, there are
$a\in(A\setminus B)\cap\B_\rho(\bx)$, $b\in(B\setminus A)\cap\B_\rho(\bx)$, $x\in\B_\rho(\bx)$ with $\norm{x-a}=\norm{x-b}$, $x^*_1,x^*_2\in X^*$, and normals $u^*_1\in N_{A}(a)$, $u^*_2\in N_{B}(b)$ satisfying \eqref{P8P-3'} and
\begin{gather}\label{P8P-51}
\|x^*_1\|+\|x^*_2\|=1,\quad \|x^*_1+x^*_2\|<\al,\quad \|x^*_1-u^*_1\|<\frac{\rho}{2},\quad \|x^*_2-u^*_2\|<\frac{\rho}{2}.
\end{gather}
Thus, $(u^*_1,u^*_2)\in\sd i_{A\times B}(a,b)$ and $(-x_1^*,-x_2^*,x^*)\in \sd f(a,b,x)$, where  $x^*:=x^*_1+x^*_2$.
By the convex sum rule (Lemma~\ref{SR}(ii)), $(w_1^*,w_2^*,x^*)\in \sd\hat f(a,b,x)$, where $w^*_1=u_1^*-x_1^*$,
$w^*_2=u_2^*-x_2^*$.
Then $\|w_1^*\|+\|w_2^*\|<\rho$
and, in view of the second representation in \eqref{str1'}, ${\rm str}_1[A,B](\bar x)\le\al$.
By letting $\alpha\downarrow\itrd{c}$, we obtain the inequality ${\rm str}_1[A,B](\bar x)\le \itrd{c}$.
\qed\end{proof}

\begin{remark}
The inequality `$\le$' in both representations in \eqref{str1'} as well as the opposite inequalities in the convex case are valid in arbitrary (not necessarily complete) normed linear spaces.
\xqed\end{remark}

\begin{proof} \emph{of Theorems~\ref{T1} and \ref{T2}}
The theorems follow now from Propositions~\ref{P7} and \ref{P8} and definitions \eqref{itrw} and \eqref{itrc}.
\qed\end{proof}

\begin{proposition}\label{P9}
Suppose $X$ is a Banach space, $A,B\subset X$ are closed and convex, and $\bx\in A\cap B$.
Then
$\str={\rm str}_1[A,B](\bar x)=\itrd{c}.$
\end{proposition}
\begin{proof}
The assertion is a consequence of Proposition~\ref{P7}(ii) and Proposition~\ref{P8}(iii).
\qed\end{proof}

\begin{remark}
Using the representations in Propositions~\ref{P4}, \ref{P5}, \ref{P6} and \ref{P8}, one can formulate several intermediate sufficient (and in some cases also necessary) conditions of subtransversality.
\xqed\end{remark}

\section{Intrinsic transversality}\label{s:intrinsic trans}

The two-limit definition \eqref{itrw} as well as the corresponding dual space characterization of subtransversality in Theorem~\ref{T1} look complicated and difficult to verify.
The following one-limit modification of \eqref{itrw} in terms of Fr\'echet normals can be useful:
\begin{gather}\label{itr}
\itr:=\liminf\limits_{\substack{
a\to\bar x,\; b\to\bar x,\; x\to\bar x\\
a\in A\setminus B,\;b\in B\setminus A,\;x\ne a,\;x\ne b\\ x^*_1\in N_{A}(a)\setminus\{0\},\;x^*_2\in N_{B}(b)\setminus\{0\},\;\norm{x^*_1}+\norm{x^*_2}=1
\\
\frac{\norm{x-a}}{\norm{x-b}}\to1,\; \frac{\ang{x^*_1,x-a}}{\norm{x^*_1}\norm{x-a}}\to1,\; \frac{\ang{x^*_2,x-b}}{\norm{x^*_2}\norm{x-b}}\to1}}
\|x^*_1+x^*_2\|,
\end{gather}
with the convention that the infimum over the empty set equals 1.
The relationship between the constants \eqref{itrw}, \eqref{itrc} and \eqref{itr} is given by the next proposition.

\begin{proposition}\label{P10}
Suppose $X$ is a Banach space, $A,B\subset X$ are closed, and $\bx\in A\cap B$.
\begin{enumerate}
\item
$0\le\itr\le\itrd{w}\le\itrd{c}\le1$;
\item
if $\dim X<\infty$, then
\begin{gather}\label{itr"}
\itrd{w}= \liminf_{\substack{a\to\bx,\;b\to\bx,\; x\to\bx\\a\in A\setminus B,\;b\in B\setminus A,\; \norm{x-a}=\norm{x-b}\\ d(x_1^*,\overline{N}_{A}(a))\to0,\; d(x_2^*,\overline{N}_{B}(b))\to0,\; \|x_1^*\|+\|x_2^*\|=1
\\
\langle x_1^*,x-a\rangle=\|x_1^*\|\,\|x-a\|,\;
\langle x_2^*,x-b\rangle=\|x_2^*\|\,\|x-b\|}}
\|x^*_1+x^*_2\|,
\end{gather}
with the convention that the infimum over the empty set equals 1;
\item
if $\dim X<\infty$, and $A$ and $B$ are convex, then $\itrd{w}=\itrd{c}=\str$.
\end{enumerate}
\end{proposition}

\begin{proof}
(i)
All three constants are nonnegative by definition and, thanks to the conventions made, never greater than 1.
Definition \eqref{itrc} corresponds to taking $x'=x$, $x_1'=a'=a$ and $x_2'=b'=b$ under the $\liminf$ in \eqref{itrw}.
Hence, $\itrd{w}\le\itrd{c}$.

Next we show that $\itr\le\itrd{w}$.
Let $\itrd{w}<\al<1$ and $\rho>0$.
Choose an $\al'$ with $\itrd{w}<\al'<\al$ and a $\rho'>0$ with
\begin{gather}\label{P10P-3}
\rho'<\min\left\{\frac{\rho}{2},\frac{1}{2}, \frac{\al-\al'}{4}, \frac{\rho(1-\al)}{4}\right\}.
\end{gather}
By definition \eqref{itrw}, there exist $a\in(A\setminus B)\cap\B_{\rho'}(\bx)$, $b\in(B\setminus A)\cap\B_{\rho'}(\bx)$ and $x\in\B_{\rho'}(\bx)$
such that
$\norm{x-a}=\norm{x-b}$,
and
\begin{gather}\label{P10P-1}
\liminf_{\substack{x'\to x,\;x_1'\to a,\;x_2'\to b,\; a'\to a,\;b'\to b\\a'\in A,\;b'\in B,\; \|x'-x_1'\|=\|x'-x_2'\|\\ d(x_1^*,N_{A}(a'))<\rho',\;d(x_2^*,N_{B}(b'))<\rho',\; \|x_1^*\|+\|x_2^*\|=1
\\
\langle x_1^*,x'-x_1'\rangle=\|x_1^*\|\,\|x'-x_1'\|,\;
\langle x_2^*,x'-x_2'\rangle=\|x_2^*\|\,\|x'-x_2'\|}}
\|x^*_1+x^*_2\|<\al'.
\end{gather}
We obviously have $a\ne b$, $x\ne a$ and $x\ne b$.
Choose an $\eps>0$ such that
\begin{gather*}
\eps<d(a,B),\quad \eps<d(b,A), \quad 2\eps\left(1+\frac{4}{1-\al}\left(\rho-\frac{4\rho'}{1-\al}\right)\iv\right) <\norm{x-a},\quad
\frac{4\eps}{\norm{x-a}-2\eps}<\rho,
\\
\|x-\bar x\|+
\varepsilon<\rho',\quad \|a-\bx\|+\varepsilon<\rho',\quad \|b-\bx\|+\varepsilon<\rho'.
\end{gather*}
By \eqref{P10P-1}, there are points $a'\in A\cap\B_\eps(a)$, $b'\in B\cap\B_\eps(b)$, $x_1'\in\B_\eps(a)$, $x_2'\in\B_\eps(b)$, $x'\in\B_\eps(x)$, and $x_1^*,x_2^*\in X^*$ satisfying conditions \eqref{T1-1}, \eqref{T1-2},
\begin{gather}\label{P10P-2}
d(x_1^*,N_{A}(a'))<\rho',\quad d(x_2^*,N_{B}(b'))<\rho' \qdtx{and}\|x^*_1+x^*_2\|<\alpha'.
\end{gather}
Then
\begin{gather}\label{P10P-16}
d(a',B)\ge d(a,B)-\norm{a'-a}>d(a,B)-\eps>0,
\\\label{P10P-17}
d(b',A)\ge d(b,A)-\norm{b'-b}>d(b,A)-\eps>0,
\\\notag
\norm{x'-a'} \le\norm{x-a}+\norm{x'-x}+\norm{a'-a} <\norm{x-a}+2\eps,
\\\notag
\norm{x'-b'} \le\norm{x-b}+\norm{x'-x}+\norm{b'-b} <\norm{x-b}+2\eps,
\\\label{P10P-8}
\norm{x'-a'} \ge\norm{x-a}-\norm{x'-x}-\norm{a'-a} >\norm{x-a}-2\eps > \frac{8\eps}{1-\alpha} \left(\rho-\frac{4\rho'}{1-\al}\right)\iv>0,
\\\label{P10P-12}
\norm{x'-b'} \ge\norm{x-b}-\norm{x'-x}-\norm{b'-b} >\norm{x-b}-2\eps > \frac{8\eps}{1-\alpha} \left(\rho-\frac{4\rho'}{1-\al}\right)\iv>0,
\\\label{P10P-20}
\frac{\norm{x'-a'}}{\norm{x'-b'}} <\frac{\norm{x-a}+2\eps}{\norm{x-b}-2\eps} =1+\frac{4\eps}{\norm{x-a}-2\eps}<1+\rho,
\\\label{P10P-21}
\frac{\norm{x'-a'}}{\norm{x'-b'}} >\frac{\norm{x-a}-2\eps}{\norm{x-b}+2\eps} =1-\frac{4\eps}{\norm{x-a}+2\eps}<1-\rho,
\\\label{P10P-13}
\norm{x'-\bx}\le\norm{x-\bx}+\norm{x'-x} <\norm{x-\bx}+\eps<\rho'<\rho,
\\\label{P10P-14}
\norm{a'-\bx}\le
\norm{a-\bx}+\norm{a'-a} < \norm{a-\bx}+\eps 
<\rho'<\rho,
\\\label{P10P-15}
\norm{b'-\bx}\le
\norm{b-\bx}+\norm{b'-b} <\norm{b-\bx}+\eps 
<\rho'<\rho,
\\\notag
\big|\|x^*_1\|-\|x^*_2\|\big|\le\|x^*_1-(-x^*_2)\|<\alpha'.
\end{gather}
The last estimate together with the equality
$\|x^*_1\|+\|x^*_2\|=1$ yield
\begin{gather}\label{P10P-5}
\|x_1^*\|
<\frac{1+\al'}{2},
\quad
\|x_2^*\|
<\frac{1+\al'}{2},
\\\label{P10P-4}
\|x_1^*\|>\frac{1-\al'}{2}>0,
\quad
\|x_2^*\|>\frac{1-\al'}{2}>0.
\end{gather}
By \eqref{P10P-2}, there are Fr\'echet normals $v_1^*\in N_{A}(a')$ and $v_2^*\in N_{B}(b')$ such that
\begin{gather}\label{P10P-6}
\norm{x_1^*-v_1^*}<\rho',\quad \norm{x_2^*-v_2^*}<\rho'.
\end{gather}
Hence, by \eqref{P10P-6}, \eqref{P10P-3}, \eqref{P10P-2}, \eqref{P10P-4} and \eqref{P10P-5},
\begin{gather}\label{P10P-9}
\norm{v_1^*}+\norm{v_2^*}\ge\norm{x_1^*}+\norm{x_2^*} -\norm{x_1^*-v_1^*}-\norm{x_2^*-v_2^*}>1-2\rho'>0,
\\\label{P10P-10}
\norm{v_1^*+v_2^*} \le\norm{x_1^*+x_2^*}+\norm{x_1^*-v_1^*}+\norm{x_2^*-v_2^*} <\al'+2\rho',
\\\label{P10P-11}
\norm{v_1^*}>\norm{x_1^*}-\rho' >\frac{1-\al'}{2}-\rho'>\frac{1-\al}{2},\quad
\norm{v_2^*}>\norm{x_2^*}-\rho' >\frac{1-\al'}{2}-\rho'>\frac{1-\al}{2},
\\\label{P10P-7}
\norm{v_1^*}<\norm{x_1^*}+\rho' <\frac{1+\al'}{2}+\rho'<\frac{1+\al}{2},\quad
\norm{v_2^*}<\norm{x_2^*}+\rho' <\frac{1+\al'}{2}+\rho'<\frac{1+\al}{2},
\end{gather}
and
\begin{align*}
\langle v_1^*,x'-a'\rangle
&\ge\langle x_1^*,x'-a'\rangle -\|x_1^*-v_1^*\|\|x'-a'\|
\\
&\ge\langle x_1^*,x'-x_1'\rangle-\|x_1^*\|\|x_1'-a'\| -\|x_1^*-v_1^*\|\|x'-a'\|
\\
&=\|x_1^*\|\|x'-x_1'\|-\|x_1^*\|\|x_1'-a'\| -\|x_1^*-v_1^*\|\|x'-a'\|&& (\text{by \eqref{T1-2}})
\\
&\ge\|v_1^*\|\|x'-x_1'\|-\|x_1^*-v_1^*\|\|x'-x_1'\|
\\
&-\|x_1^*\|\|x_1'-a'\| -\|x_1^*-v_1^*\|\|x'-a'\|
\\
&\ge\|v_1^*\|(\|x'-a'\|-\|x_1'-a'\|)
\\
&-\|x_1^*-v_1^*\|(\|x'-a'\|+\|x_1'-a'\|)
\\
&-\|x_1^*\|\|x_1'-a'\| -\|x_1^*-v_1^*\|\|x'-a'\|
\\
&=(\|v_1^*\|-2\|x_1^*-v_1^*\|)\|x'-a'\|
\\
&-(\|v_1^*\|+\|x_1^*-v_1^*\|+\|x_1^*\|)\|x_1'-a'\|
\\
&>(\|v_1^*\|-2\rho')\|x'-a'\| - \left(\frac{1+\al}{2}+\frac{\al-\al'}{2} +\frac{1+\al'}{2}\right)2\eps&& (\text{by \eqref{P10P-6}, \eqref{P10P-7}, \eqref{P10P-5}, \eqref{P10P-3}})
\\
&=(\|v_1^*\|-2\rho')\|x'-a'\| -(1+\al)2\eps
\\
&>(\|v_1^*\|-2\rho')\|x'-a'\| -4\eps && (\al<1)
\\
&>\left(\|v_1^*\|-2\rho' -\frac{1-\alpha}{2}\rho+2\rho'\right) \|x'-a'\|
&& (\text{by \eqref{P10P-8}})
\\
&=\left(\|v_1^*\| -\frac{1-\alpha}{2}\rho\right) \|x'-a'\|
\\
&>\|v_1^*\|(1-\rho)\|x'-a'\|.
&& (\text{by \eqref{P10P-11}})
\end{align*}
Thus,
\begin{gather*}
\frac{\ang{v_1^*,x'-a'}}{\|v_1^*\|\|x'-a'\|} >1-\rho.
\end{gather*}
Similarly,
\begin{gather*}
\frac{\ang{v_2^*,x'-b'}}{\|v_2^*\|\|x'-b'\|} >1-\rho.
\end{gather*}
Set
\begin{gather*}
\hat x^*_1 =\frac{v_1^*}{\left\|v_1^*\right\|+\left\|v_2^*\right\|}, \quad
\hat x^*_2 =\frac{v_2^*}{\left\|v_1^*\right\|+\left\|v_2^*\right\|}.
\end{gather*}
Then $\hat x^*_1\in N_{A}(a')\setminus\{0\}$, $\hat x^*_2\in N_{B}(b')\setminus\{0\}$, $\left\|\hat x^*_1\right\|+\left\|\hat x^*_2\right\|=1$ and, by \eqref{P10P-9}, \eqref{P10P-10}, \eqref{P10P-3} and
the inequality $1+\al<2$, we have
\begin{gather*}
\left\|\hat x^*_1+\hat x^*_2\right\| =\frac{\|v_1^*+v_2^*\|} {\left\|v_1^*\right\|+\left\|v_2^*\right\|} <\frac{\al'+2\rho'}{1-2\rho'} <\frac{\al'+\frac{\al-\al'}{2}} {1-\frac{\al-\al'}{2}} <\frac{\al'+\frac{\al-\al'}{1+\al}} {1-\frac{\al-\al'}{1+\al}}=\al,
\\
\frac{\ang{\hat x_1^*,x'-a'}} {\|\hat x_1^*\|\|x'-a'\|}>1-\rho,
\quad
\frac{\ang{\hat x_2^*,x'-b'}} {\|\hat x_2^*\|\|x'-b'\|}>1-\rho.
\end{gather*}
Hence, recalling \eqref{P10P-16}, \eqref{P10P-17}, \eqref{P10P-8}, \eqref{P10P-12}, \eqref{P10P-13}, \eqref{P10P-14} and \eqref{P10P-15},
\begin{gather*}
\inf\limits_{\substack{a'\in(A\setminus B)\cap\B_\rho(\bx),\; b'\in(B\setminus A)\cap\B_\rho(\bx)\\
x'\in\B_\rho(\bx),\;x'\ne a',\;x'\ne b'\\
\hat x^*_1\in N_{A}(a')\setminus\{0\},\;\hat x^*_2\in N_{B}(b')\setminus\{0\},\;\norm{\hat x^*_1}+\norm{\hat x^*_2}=1
\\
1-\rho<\frac{\norm{x'-a'}}{\norm{x'-b'}}<1+\rho,\; \frac{\ang{\hat x_1^*,x'-a'}} {\|\hat x_1^*\|\|x'-a'\|}>1-\rho,\;
\frac{\ang{\hat x_2^*,x'-b'}} {\|\hat x_2^*\|\|x'-b'\|}>1-\rho}}
\|\hat x^*_1+\hat x^*_2\|<\al,
\end{gather*}
The claimed inequality follows after passing to the limits as $\rho\downarrow0$ and $\alpha\downarrow\itrd{w}$.

(ii) If $\dim X<\infty$, then, thanks to the compactness of the unit sphere, the $\liminf$ in \eqref{itrw} reduces to
\begin{gather*}
\inf_{\substack{ d(x_1^*,\overline{N}_{A}(a))\le\rho,\; d(x_2^*,\overline{N}_{B}(b))\le\rho,\; \|x_1^*\|+\|x_2^*\|=1
\\
\langle x_1^*,x-a\rangle=\|x_1^*\|\,\|x-a\|,\;
\langle x_2^*,x-b\rangle=\|x_2^*\|\,\|x-b\|}}
\|x^*_1+x^*_2\|.
\end{gather*}
As a result, the \RHS\ of \eqref{itrw} reduces to that of \eqref{itr"}.

(iii) In the convex case, the limiting and Fr\'echet normal cones coincide, and so do the \RHS s of \eqref{itr"} and \eqref{itrc}.
The second equality is a consequence of Proposition~\ref{P9}.
\qed\end{proof}

The property introduced in Theorem~\ref{T1} as a sufficient dual space characterization of subtransversality and corresponding to the condition $\itrd{w}>0$ as well as the stronger property corresponding to the condition $\itr>0$ are themselves important transversality properties of the pair $\{A,B\}$ at $\bx$.
Borrowing partially the terminology from \cite{DruIofLew15.2}, we are going to call these properties \emph{weak intrinsic transversality} and \emph{intrinsic transversality}, respectively.
\begin{definition}\label{D3+}
Suppose $X$ is a normed linear space, $A,B\subset X$ are closed, and $\bx\in A\cap B$.
The pair $\{A,B\}$ is
\begin{enumerate}
\item
weakly intrinsically transversal at $\bar x$ if $\itrd{w}>0$, i.e., there exist numbers $\alpha\in]0,1[$ and $\delta>0$ such that, for all $a\in(A\setminus B)\cap\B_\de(\bx)$, $b\in(B\setminus A)\cap\B_\de(\bx)$ and $x\in\B_\de(\bx)$
with $\norm{x-a}=\norm{x-b}$,
one has $\|x^*_1+x^*_2\|>\alpha$ for some $\eps>0$ and all $a'\in A\cap\B_\eps(a)$, $b'\in B\cap\B_\eps(b)$, $x_1'\in\B_\eps(a)$, $x_2'\in\B_\eps(b)$, $x'\in\B_\eps(x)$, and $x_1^*,x_2^*\in X^*$ satisfying conditions \eqref{T1-1}, \eqref{T1-2} and \eqref{T1-3};
\item
intrinsically transversal at $\bar x$ if $\itr>0$, i.e., there exist numbers $\alpha\in]0,1[$ and $\delta>0$ such that $\|x^*_1+x^*_2\|>\alpha$
for all $a\in(A\setminus B)\cap\B_\de(\bx)$, $b\in(B\setminus A)\cap\B_\de(\bx)$, $x\in\B_\de(\bx)$, $x_1^*\in N_{A}(a)\setminus\{0\}$ and $x_2^*\in N_{B}(b)\setminus\{0\}$ satisfying
\begin{gather}\label{D3+1}
x\ne a,\quad x\ne b,\quad 1-\delta<\frac{\norm{x-a}}{\norm{x-b}}<1+\delta,
\\\label{D3+2}
\norm{x^*_1}+\norm{x^*_2}=1,\quad\frac{\ang{x_1^*,x-a}} {\|x_1^*\|\|x-a\|} >1-\delta,\quad
\frac{\ang{x_2^*,x-b}} {\|x_2^*\|\|x-b\|} >1-\delta.
\end{gather}
\end{enumerate}
\end{definition}

\begin{remark}\label{R7}
The properties introduced in Definition~\ref{D3+} are less restrictive than the dual criterion of transversality in Theorem~\ref{T0}.
\xqed\end{remark}

In view of Definition~\ref{D3+}, Theorem~\ref{T1} says that in Asplund spaces weak intrinsic transversality (and consequently intrinsic transversality) implies subtransversality.
Thanks to Proposition~\ref{P10}(i) and Remark~\ref{R7}, we have the following chain of implications in Asplund spaces:
\begin{align*}
\text{transversality}
\quad&\Longrightarrow\quad
\text{intrinsic transversality}
\\
&\Longrightarrow\quad
\text{weak intrinsic transversality}
\quad\Longrightarrow\quad
\text{subtransversality}.
\end{align*}
By Proposition~\ref{P10}(iii),
when the space is finite dimensional and the sets are convex, the last two properties are equivalent.

As a consequence of Proposition~\ref{P10}(i), we obtain the following dual sufficient condition of subtransversality of a pair of closed sets in an Asplund space.
It expands and improves \cite[Theorem~4.1]{KruTha15} as well as a more recent result announced without proof in the Euclidean space setting in \cite[Theorem~4(ii)]{KruLukTha}.

\begin{corollary}\label{C0}
Suppose $X$ is Asplund, $A,B\subset X$ are closed, and $\bx\in A\cap B$.
Then $\{A,B\}$ is subtransversal at $\bar x$ if there exist numbers $\alpha\in]0,1[$ and $\de>0$ such that $\|x^*_1+x^*_2\|>\alpha$
for all $a\in(A\setminus B)\cap\B_\de(\bx)$, $b\in(B\setminus A)\cap\B_\de(\bx)$, $x\in\B_\de(\bx)$, $x_1^*\in N_{A}(a)\setminus\{0\}$ and $x_2^*\in N_{B}(b)\setminus\{0\}$ satisfying \eqref{D3+1} and \eqref{D3+2}.
\end{corollary}



\section*{Acknowledgements}
The authors thank the referees for the careful reading of the manuscript and constructive comments and suggestions.

\end{document}